\documentclass[twoside,12pt]{amsart}

\usepackage{amsmath,latexsym,amssymb,mathptm, times,verbatim, mathcomp}
    \usepackage{longtable}\usepackage{tikz,stmaryrd}

\usepackage{enumerate}

\input amssym.def
\input amssym
\input xypic
\input xy
\xyoption{all}
\def\Kos{\operatorname{Kos}}

\def\modact{\operatorname{ModAct}}
\def\mult{\operatorname{mult}}
\def\id{\operatorname{id}}

\def\Spec{\operatorname{Spec}}
\def\Ass{\operatorname{Ass}}
\def\Supp{\operatorname{Supp}}
\def\Ext{\operatorname{Ext}}
\def\Tor{\operatorname{Tor}}
\def\im{\operatorname{im}}

\def\length{\operatorname{length}}

\def\rank{\operatorname{rank}}
\def\fakeht{\vphantom{E^{E_E}_{E_E}}}

\def\coker{\operatorname{coker}}

\define\Hom{\operatorname{Hom}}

\def\Sym{\operatorname{Sym}}

\def\ann{\operatorname{ann}}
\def\Kos{{\mathrm {Kos}}}

\def\incl{\operatorname{incl}}

\def\grade{\operatorname{grade}}
\def\pd{\operatorname{pd}}
\def\Tot{\operatorname{Tot}}
\def\HH{\operatorname{H}}
\def\ev{\operatorname{ev}}

\def\quot{\operatorname{quot}}

\newtheorem{theorem}{Theorem}[section]
\newtheorem{lemma}[theorem]{Lemma}
\newtheorem{corollary}[theorem]{Corollary}
\newtheorem{proposition}[theorem]{Proposition}

\newtheorem{observation}[theorem]{Observation}

\theoremstyle{definition}
\newtheorem{definition}[theorem]{Definition}
\newtheorem{remark}[theorem]{Remark}
\newtheorem{remark-no-advance}[equation]{Remark}
\newtheorem{remarks-no-advance}[equation]{Remarks}
\newtheorem{claim-no-advance}[equation]{Claim}
\newtheorem{proposition-no-advance}[equation]{Proposition}
\newtheorem{corollary-no-advance}[equation]{Corollary}

\newtheorem{remarks}[theorem]{Remarks}
\newtheorem{data}[theorem]{Data}

\newtheorem{notation}[theorem]{Notation}

\newtheorem{present summary}[theorem]{Present Summary}

\newtheorem{example}[theorem]{Example}

\newtheorem{chunk}[theorem]{}
\newtheorem{subchunk}[equation]{}
\newtheorem{marching orders}[theorem]{Marching Orders}

\newtheorem{circle the wagons}[theorem]{Circle the wagons}

\newtheorem*{Remark}{Remark}

\newtheorem*{proof-of-claim-1}{Proof of Claim~\ref{claim1}}
\newtheorem*{proof-of-claim-2}{Proof of Claim~\ref{claim2}}
\newtheorem*{proof-of-key}{Proof of Lemma~\ref{key}}
\newtheorem*{proof-of-claim-3}{\rm Proof of Claim~\ref{claim3}}

\numberwithin{equation}{theorem}
\numberwithin{table}{theorem}
\numberwithin{figure}{theorem}

\begin{document}

\baselineskip=16pt

\title[Canonical complexes associated to a matrix]{\bf  
Canonical complexes associated to a matrix}

\date\today 
\author[Andrew R. Kustin]
{Andrew R. Kustin}

\thanks{AMS 2010 {\em Mathematics Subject Classification}.
Primary 13D02; Secondary 13C40.}

\thanks{The 
author was partially supported by the the Simons Foundation
grant number 233597.}


\address{Department of Mathematics, University of South Carolina,
Columbia, SC 29208} \email{kustin@math.sc.edu}

\begin{abstract}
Let $\Phi$ be an  $\goth f\times \goth g$ matrix with entries from a commutative Noetherian 
 ring $R$, with $\goth g\le \goth f$.
Recall the family of generalized Eagon-Northcott complexes $\{\mathcal C^{i}_\Phi\}$ associated to $\Phi$. (See, for example, Appendix A2 in ``Commutative Algebra with a view toward Algebraic Geometry'' by D.~Eisenbud.) For each integer $i$,  
$\mathcal C^i_\Phi$  
is a complex of free $R$-modules. For example,  $\mathcal C^{0}_\Phi$ is the original ``Eagon-Northcott'' complex with zero-th homology equal to the ring $R/I_\goth g(\Phi)$ defined by ideal generated by the maximal order minors of $\Phi$; and  $\mathcal C^{1}_\Phi$ is the ``Buchsbaum-Rim'' complex with zero-th homology equal to the cokernel of the transpose of $\Phi$. If $\Phi$ is sufficiently general, then each $\mathcal C^{i}_\Phi$, with $-1\le i$, is acyclic; and, if $\Phi$ is generic, then these complexes resolve half of the divisor class group of $R/I_\goth g(\Phi)$. The family $\{\mathcal C^{i}_\Phi\}$ exhibits duality; and, if $-1\le i\le \goth f-\goth g+1$, then the complex $\mathcal C^{i}_\Phi$ exhibits depth-sensitivity with respect to the ideal $I_{\goth g}(\Phi)$ in the sense that the tail of $\mathcal C^{i}_\Phi$ of length equal to $\grade(I_{\goth g}(\Phi))$ is acyclic. The entries in the differentials of $\mathcal C^i_\Phi$ are linear in the entries of $\Phi$ at every position except  at  one, where the 
entries of the differential are $\goth g\times \goth g$ minors of $\Phi$.

This paper expands the family $\{\mathcal C^i_\Phi\}$ to a family of complexes
$\{\mathcal C^{i,a}_\Phi\}$ for integers $i$ and $a$ with $1\le a\le \goth g$.
The entries in the differentials of $\{\mathcal C^{i,a}_\Phi\}$ are linear in the entries of $\Phi$ at every position except at  two consecutive positions. At one of the  exceptional positions the entries are $a\times a$ minors of $\Phi$, at the other exceptional position the entries are $(\goth g-a+1)\times (\goth g-a+1)$ minors of $\Phi$. 

The complexes $\{\mathcal C^i_\Phi\}$ are equal to 
$\{\mathcal C^{i,1}_\Phi\}$ and $\{\mathcal C^{i,\goth g}_\Phi\}$. The complexes $\{\mathcal C^{i,a}_\Phi\}$ exhibit all of the properties of $\{\mathcal C^{i}_\Phi\}$. In particular, if $-1\le i\le \goth f-\goth g$ and $1\le a\le \goth g$, then $\mathcal C^{i,a}_\Phi$ exhibits depth-sensitivity with respect to the ideal $I_{\goth g}(\Phi)$.
\end{abstract}

\maketitle

\tableofcontents

\bigskip

\section{Introduction.} 

\bigskip

Let $R$ be a commutative Noetherian ring and  $F$ and $G$ be free $R$-modules of rank $\goth f$ and $\goth g$, respectively, with $\goth g\le \goth f$.
Recall that, for each $R$-module homomorphism \begin{equation}\label{Phi}\Phi: G^*\to F\end{equation}
there is a family of generalized Eagon-Northcott complexes $\{\mathcal C_{\Phi}^{i}\}$. (See, for example, Definition~\ref{FFR}, \cite[Appendix A.2]{Ei95},    \cite[2.16]{BV}, or \cite{Ki74}. A more complete history of these complexes may be found in the comments on page 26 in \cite{BV}.) If \begin{equation}\label{8acyclic-range}-1\le i\le \goth f -\goth g+1,\end{equation} then  $\mathcal C_{\Phi}^{i}$ has length $\goth f -\goth g+1$; and, if $\goth f -\goth g+1\le \grade I_{\goth g}(\Phi)$, then   $\mathcal C_{\Phi}^{i}$  is acyclic for $i$ satisfying (\ref{8acyclic-range}). Furthermore, the complexes  $\mathcal C_{\Phi}^{i}$, for $i$ satisfying (\ref{8acyclic-range}), exhibit depth-sensitivity. In particular, if $s\le \grade I_{\goth g}(\Phi)$ for some integer $s$ with  $0\le s\le f-g+1$, then $\HH_j(\mathcal C_{\Phi}^i)=0$ for $\goth f-\goth g+2-s\le j$ and $i$ satisfying  (\ref{8acyclic-range}). 
In the generic situation, the complexes $\mathcal C_{\Phi}^i$, with $i$ satisfying (\ref{8acyclic-range}), resolve the Cohen-Macaulay elements of the divisor class group of the determinantal ring $R/I_{\goth g}(\Phi)$. 
The complexes $\mathcal C_{\Phi}^i$, with $i$ in the range (\ref{8acyclic-range}), exhibit duality:
$$\mathcal C_{\Phi}^i \cong \text{a shift of $\Hom_R(\mathcal C_{\Phi}^j,R)$ in homological degree},$$
for $i+j=\goth f-\goth g$. Also, if $R$ is a graded ring, and a matrix representation of $\Phi$ is a matrix of linear forms, then the Betti tables for the complexes $\{\mathcal C_{\Phi}^i\}$ are pleasing to the eye. The maps are linear, except at, at most one position where the maps have degree $\goth g$. Moreover,  the  position of non-linearity slides, along a line of slope $1$, from the beginning of the complex to the end as $i$ varies from $0$ to $\goth f-\goth g$.

We expand the list of canonical complexes which are associated to the    $R$-module homomorphism (\ref{Phi}). For each pair $(i,a)$ with
\begin{equation}\label{iab} -1\le i\le \goth f-\goth g \quad \text{and} \quad 1\le a\le \goth g,\end{equation} we consider a complex $\mathcal C_{\Phi}^{i,a}$. The classical generalized Eagon-Northcott complexes 
\begin{equation}\label{ClGEN}\{\mathcal C_{\Phi}^{i}| \text{(\ref{8acyclic-range}) holds}\}\end{equation}
are included in the set
\begin{equation}\label{GENGEN}\{\mathcal C_{\Phi}^{i,a}| \text{(\ref{iab}) holds}\}\end{equation} with $\mathcal C_{\Phi}^{i}=
\mathcal C_{\Phi}^{i,1}$ for $-1\le i\le \goth f-\goth g$ and $\mathcal C_{\Phi}^{i}=
\mathcal C_{\Phi}^{i-1,\goth g}$ for $0\le i\le \goth f-\goth g+1$.
 The complexes of (\ref{GENGEN}) exhibit many of the properties as the listed properties for the generalized Eagon-Northcott  complexes (\ref{ClGEN}). Each complex of (\ref{GENGEN})  has length $\goth f -\goth g+1$; and,
$$\text{if $\goth f -\goth g+1\le \grade I_{\goth g}(\Phi)$, then  each  $\mathcal C_{\Phi}^{i,a}$  is acyclic.}$$ The complexes of  (\ref{GENGEN}) exhibit depth-sensitivity. In the generic situation, the complex $\mathcal C_{\Phi}^{i,a}$ of (\ref{GENGEN}) \begin{equation}\label{promise}\textstyle \text{resolves a maximal  Cohen-Macaulay module over the ring $R/I_{\goth g}(\Phi)$ of rank $\binom{\goth g-1}{a-1}$.}\end{equation} 
The complexes of (\ref{GENGEN})  exhibit duality:
$$\mathcal C_{\Phi}^{i,a} \cong \text{a shift of $\Hom_R(\mathcal C_{\Phi}^{j,b},R)$ in homological degree},$$ for $i+j=\goth f-\goth g-1$ and $a+b=\goth g+1$. Also, if $R$ is a graded ring, and  $\Phi$ is a map of degree $1$, then 
the maps of $C_{\Phi}^{i,a}$  are linear, except at, at most two adjacent  positions where the maps have degree $a$ and $\goth g+1-a$. Moreover,  the  position of non-linearity slides, along a line of slope $1$, from the beginning of the complex to the end as $i$ varies from $-1$ to $\goth f-\goth g$; see Example~\ref{3.3}.

The Eagon-Northcott \cite{EN} complex $\mathcal C^0_{\Phi}$ and the Buchsbaum-Rim complex \cite{BR-1,BR-2,BR-3} $\mathcal C^1_{\Phi}$ are very important objects in Commutative Algebra and Algebraic Geometry. (For example, \cite{B-A04,G96,HH,KT94,UV} is a small sampling of the recent papers about Buchsbaum-Rim multiplicity and its application to equisingularity.) We expect that the rest of the family (\ref{GENGEN}) will prove to be valuable tools in these fields.

The complexes $\mathcal C^{i,a}_\Phi$ arise in the study of the homological properties of the primary components of the content ideal $c(fgh)$ of the product of three generic polynomials $f$, $g$, and $h$. These components have been identified \cite[Thm.~4.2]{CVV} and all but one of the components is known to be Gorenstein \cite[Thm.~4.1 and Rem.~4.3]{CVV}.  
The complexes $\mathcal C^{i,a}_\Phi$ also arise 
in the study of 
the resolutions of the symmetric algebra $\Sym(I)$ and the Rees algebra $\mathcal R(I)$ of a grade three Gorenstein ideal $I=(g_1,\dots,g_n)$ in a polynomial ring over a field $k$; see, for example, \cite{KPU-Ann} and \cite[Cor.~6.3]{KPU-BA}, where the special fiber ring $\mathcal F(I)=k[g_1,\dots,g_n]$ of $I$ is resolved.

The complexes $\mathcal C^{i,a}_\Phi$ are  straightforward and they are built in a canonical manner. That is, there are no choices; everything  is coordinate-free. The modules in $\mathcal C^{i,a}_\Phi$ are Schur modules and Weyl modules corresponding to hooks. In other words, the modules in $\mathcal C^{i,a}_\Phi$ all are kernels of Koszul complex maps or Eagon-Northcott complex maps; see \ref{2.3} and \ref{4.6}. The complex $\mathcal C^{i,a}_\Phi$ is obtained by concatenating three finite complexes:
$$\mathbb K \to \bigwedge \to \mathbb L,$$
where  $\mathbb K$ and $\mathbb L$ are standard complexes of Weyl and Schur modules, respectively, and $\bigwedge$ consists of a single exterior power concentrated in one position; see Definition~\ref{CiaPhi} for the details. The complexes $\mathbb K$ and $\mathbb L$ are introduced in Section~\ref{K-and-L}. 

The main result of this paper is Theorem~\ref{main} which states that
if $\Phi$ is sufficiently general, $-1\le i$, and $1\le a\le \goth g$, then $\mathcal C^{i,a}_\Phi$ is an  acyclic complex of free $R$-modules  and $\HH_0(\mathcal C^{i,a}_\Phi)$ is a torsion-free  $R/I_\goth g(\Phi)$-module of rank $\binom{\goth g-1}{a-1}$. 
The most important applications occur when $i$ also satisfies $i\le \goth f-\goth g$. Indeed, in this situation, $\mathcal C^{i,a}_\Phi$  has length $\goth f-\goth g+1$ and,
if 
$\goth f-\goth g+1\le \grade I_\goth g(\Phi)$, then 
$\HH_0(\mathcal C^{i,a}_\Phi)$ is a perfect $R$-module of projective dimension $\goth f-\goth g+1$ resolved by $\mathcal C^{i,a}_\Phi$  and 
$\Ext_R^{\goth f-\goth g+1}(\HH_0(\mathcal C^{i,a}_\Phi),R) $ is a perfect $R$-module resolved by $\mathcal C_\Phi^{\goth f-\goth g-i-1,\goth g+1-a}$;
furthermore, even if $\grade I_\goth g(\Phi)<\goth f-\goth g+1$,   the complex $\mathcal C^{i,a}_\Phi$ exhibits depth-sensitivity with respect to the ideal $I_\goth g(\Phi)$ in the sense that \begin{equation}\label{DS}\text{$\HH_j(\mathcal C^{i,a}_\Phi)=0$ for $\goth f-\goth g+2-\grade I_{\goth g}(\Phi)\le j$, when $-1\le i\le \goth f-\goth g$ and $1\le a\le \goth g$.}\end{equation}
The depth-sensitivity (\ref{DS}) allows one to use truncations of various  $\mathcal C^{i,a}_\Phi$ as acyclic strands in resolutions even when $I_\goth g(\Phi)$ is known to be less than $\goth f-\goth g+1$.

There is a basic similarity between the present paper and and the paper \cite{KU92},
which produces a family of complexes $\{\mathcal D^q_\rho\}$ for each almost alternating homomorphism $\rho$. The family $\{\mathcal D^q_\rho\}$ shares many properties with the family of generalized Eagon-Northcott complexes $\{\mathcal C^{i}_\Phi\}$. The main difference between \cite{KU92} and the present paper is that the homomorphism $\rho$ of \cite{KU92} is special, in the sense that it is an almost alternating homomorphism; whereas, the complexes $\{\mathcal C^{i,a}_\Phi\}$ of the present paper and the generalized  Eagon-Northcott complexes $\{\mathcal C^{i}_\Phi\}$
are both
constructed from an arbitrary homomorphism $\Phi$.
Nonetheless, the statement of Theorem~\ref{main} and  some steps in its proof are  modeled on \cite[Thm.~8.3]{KU92}, although \cite{KU92} does not contain any analogue to Lemma~\ref{key}, which is the key calculation in the present paper. In place of a result like Lemma~\ref{key}, \cite{KU92} first treats the generic case and then specializes to the non-generic case. In the present paper, Lemma~\ref{key} shows that if the image of $\Phi$ contains a basis element of $F$, then 
$$\mathcal C_\Phi^{i,a} \quad \text{and} \quad \mathcal C_{\Phi'}^{i,a}\oplus \mathcal C_{\Phi'}^{i,a-1}$$ have isomorphic homology for some (``smaller'') 
$R$-module homomorphism $\Phi'$. To prove Theorem~\ref{main}, we iterate Lemma~\ref{key} and apply the  acyclicity lemma.  The representation theory that is used in the proof of Lemma~\ref{key} is begun in 
(\ref{dentist}) and (\ref{rep-thy}) and carried out in Proposition~\ref{87.5}.

The complexes $\mathcal C^{i,a}_\Phi$ are defined in \ref{CiaPhi}; examples are given in \ref{3.3} and \ref{5.6}; the duality is treated in \ref{duality}; and the zero-th homology 
$$\HH_0(\mathcal C^{i,a}_\Phi)=\begin{cases}
\frac{\bigwedge^{\goth f-\goth g+a}F}{\im (\bigwedge^a\Phi)\wedge \bigwedge^{\goth f-\goth g}F},&\text{if $i=-1$,}\vspace{5pt}\\
\coker (\bigwedge^{\goth g-a+1}\Phi^*),&\text{if $i=0$, and}\vspace{5pt}\\
\frac{L_{i+1}^{\goth g-a}G}{\Phi^*(F)\cdot L_{i}^{\goth g-a}G},&\text{if $1\le i$.}\end{cases}$$ is calculated in \ref{H0}. The Schur module $L^p_qG$ is described in \ref{2.3} and \ref{4.6}.

\bigskip

\section{Notation, conventions, and preliminary results.}\label{Prelims}

\bigskip

There are three subsections: Ground rules, Grade and perfection, and Multilinear algebra.

\subsection{Ground rules.}

\bigskip\begin{chunk}Unless otherwise noted, $R$ is a commutative Noetherian ring and all functors are functors of $R$-modules; that is, $\otimes$, $\operatorname{Hom}$, $(\underline{\phantom{X}})^*$, $\operatorname{Sym}_i$, $D_i$, and 
$\bigwedge^i$ mean
$\otimes_R$, $\operatorname{Hom}_R$, $\operatorname{Hom}_R(\underline{\phantom{X}},R)$, $\operatorname{Sym}_i^R$, $D_i^R$,  and $\bigwedge^i_R$,  respectively.\end{chunk} 

\begin{chunk}A complex $\mathcal C: \cdots\to C_2\to C_1\to C_0\to 0$ of $R$-modules is called {\it acyclic} if $\HH_j(\mathcal C)=0$ for $1\le j$.\end{chunk}

\begin{chunk}\label{length}
If a complex $\mathcal C$ has the form
$$0\to \mathcal C_{\ell}\to \mathcal C_{\ell-1}\to \dots \to \mathcal C_1\to \mathcal C_0\to 0,$$ with $\mathcal C_0\neq 0$ and $\mathcal C_{\ell}\neq 0$,  then we say that the {\it length of $\mathcal C$} is  $\ell$ and we write $\length(\mathcal C)=\ell$.
\end{chunk}

\begin{chunk}\label{shift} If $\mathbb A$ is a complex and $n$ is an integer, then $\mathbb A[n]$ is a new complex with $\big[\mathbb A[n]\big]_j=\mathbb A_{n+j}$. \end{chunk}

\begin{chunk}\label{Tot}If $(\mathbb A,a)\xrightarrow{\theta} (\mathbb B,b)$ is a map of complexes, then the {\it total complex} (or mapping cone) of $\theta$, denoted $\Tot(\theta)$, is 
the complex $(\mathbb T,t)$ with $\mathbb T=\mathbb A[-1]\oplus \mathbb B$ as a graded module. The  differential in $\mathbb T$ is given by
$$\mathbb T_j= \begin{matrix}\mathbb  A_{j-1}\\\oplus\\\mathbb  B_j\end{matrix}\xrightarrow{\quad t_j\quad}
\begin{matrix}\mathbb  A_{j-2}\\\oplus\\\mathbb  B_{j-1}\end{matrix}=\mathbb T_{j-1},$$
with
$$t_j=\bmatrix a_{j-1}&0\\\theta_{j-1}&-b_j\endbmatrix.  
$$\end{chunk}

\begin{chunk} If $\Phi$ is a matrix (or a homomorphism of free $R$-modules), then $I_r(\Phi)$ is the ideal generated by the
$r\times r$ minors of $\Phi$ (or any matrix representation of $\Phi$). 
\end{chunk}

\begin{chunk}\label{chi}If $S$ is a statement then 
$$\chi(S)=\begin{cases} 1,&\text{if $S$ is true,}\vspace{5pt}\\0,&\text{if $S$ is false.}\end{cases}$$\end{chunk}

\bigskip
\subsection{Grade and perfection.}

\bigskip\begin{chunk}\label{gunm-jr}The {\it grade} of a proper ideal  $I$ in a Noetherian ring $R$  is the length of a maximal  $R$-regular sequence in $I$. The unit ideal $R$  of $R$ is regarded as an ideal  of infinite grade.
\end{chunk}

\begin{chunk}\label{perfect} 
Let $M$ be a non-zero finitely generated module over a Noetherian ring $R$ and let $\ann(M)$ be the annihilator of $M$ and $\pd_RM$ be the projective dimension of $M$. It is well-known that $$\grade \ann (M)=\min\{j\mid \Ext^j_R(M,R)\neq 0\};$$therefore, it follows that 
\begin{equation}\label{above}\textstyle \operatorname{grade}\ann (M)\le \operatorname{pd}_R M.\end{equation}
If equality holds in (\ref{above}), then $M$ is called a {\it perfect $R$-module}.  
Recall, for example, that if $R$ is Cohen-Macaulay and $M$ is perfect, then $M$ is Cohen-Macaulay. 
 (This is not the full story. For more information, see, for example, \cite[Prop.~16.19]{BV} or \cite[Thm.~2.1.5]{BH}.)

The ideal $I$  in  $R$ is called a {\it perfect ideal} if $R/I$ is a perfect $R$-module.
\end{chunk}

\begin{chunk}\label{rank}
Let $M$ be a finitely generated $R$-module. 
The $R$-module $M$ has {\it rank} $r$ if $M_{\mathfrak p}$ is a free 
$R_{\mathfrak p}$-module of rank $r$ for all associated primes $\mathfrak p$ of  $\Ass R$.
If every non-zero-divisor in $R$ is regular on $M$, then $M$ is called {\it torsion-free}. 
 A proof of the following result  may be found in \cite[Prop. 1.25]{KU92}.

\begin {proposition-no-advance} \label{1.25} Let $M$ be a non-zero finitely
generated
$R$-module 
with finite projective dimension. 
Suppose that $I$ is a perfect  
ideal of $R$ with
$IM = 0$. For each integer $w$, with $1 \leq w \leq \pd_RM$, let
$F_{w}$ be the ideal 
of $R$ generated by
$$ \{ x \in R \mid \pd_{R_{x}}M_{x} < w \} . $$
If $w + 1\le \grade F_{w} $ for all $w$ with $\grade I + 
1 
\leq w \leq \pd_{R}M$, 
then   $M$ is a torsion-free $R/I$ module.
In particular, if $\pd_{R}M \leq 
\grade I$, then $M$ is a torsion-free $(R/I)$-module.
 \qed \end{proposition-no-advance}
\end{chunk}

\begin{chunk} The following statement is well-known; see, for example, \cite[Cor.~6.10]{Ho75}. It follows from the fact that the if $M$ is a perfect module, then the  ideals $F_w$ for $M$ (in the sense \ref{1.25}) and the annihilator of $M$  all have the same radical.

\begin{proposition-no-advance}\label{Ho-6.10} 
Let $A\to R$ be a homomorphism of Noetherian rings and $M$ be a non-zero finitely generated perfect $A$-module. 
If $M\otimes_AR\neq 0$, then 
$$\grade((\ann M)R)=\pd_AM-\max\{i\mid \Tor^A_i(M,R)\neq 0\}. \qed$$\end{proposition-no-advance}

In our favorite applications of \ref{Ho-6.10}, we focus on the complex $\mathbb F\otimes_AR$, where $\mathbb F$ is a resolution of $M$  by projective $A$-modules. We are satisfied with an inequality; and therefore, there is no need for us to assume that $M\otimes_AR\neq 0$. Furthermore, in practice, $A$ is usually a polynomial ring over the ring of integers. We refer to the following statement as depth-sensitivity. Of course, these ideas were first worked out in \cite{EN67,N68,N70,Ho71}, and especially \cite[Cor.~3.1]{Ho74}.

\begin{proposition-no-advance}\label{depth-sensitivity-result}
Let $A\to R$ be a homomorphism of Noetherian rings,  $M$ be a non-zero finitely generated perfect $A$-module, and $\mathbb F$ be a resolution of $M$  by projective $A$-modules with the length of $\mathbb F$ equal to the projective dimension of $M$. Then 
$$\HH_j(\mathbb F\otimes _AR)=0\quad \text{for $\pd_AM -\grade((\ann M)R)+1\le j$}. \qed$$
\end{proposition-no-advance}
 It is worth observing that if $\ann M$ is replaced by any sub-ideal, then the inequality of \ref{depth-sensitivity-result} continues to hold.
\end{chunk}
  
\bigskip
\subsection{Multilinear algebra.}

 \bigskip\begin{chunk} \label{87Not1} 
Our complexes are described in a coordinate-free manner. Let $V$ be a free module of finite rank $d$ over 
$R$.  
We make much use of 
the symmetric algebra $\Sym_\bullet V$, the divided power algebra  $D _{\bullet}(V^*)$, and the exterior algebras $\bigwedge^{\bullet} V$ and $\bigwedge^{\bullet} (V^*)$. We use the fact that 
$D _{\bullet}(V^*)$ is a module over $\Sym_\bullet V$ and the fact that $\bigwedge^{\bullet} V$ and $\bigwedge^{\bullet} (V^*)$ are modules over one another. We also use the fact that these module actions give rise to natural perfect pairings 
$$\ev:\Sym_iV\otimes D_i(V^*)\to R\quad\text{and}\quad \textstyle \ev:\bigwedge^i V\otimes 
\bigwedge^i (V^*)\to R,$$ for each integer $i$. The duals 
$$\ev^*:R\to D_i(V^*)\otimes \Sym_iV\quad\text{and}\quad \textstyle \ev^*:R\to\bigwedge^i (V^*)\otimes 
\bigwedge^i V$$ of the above evaluation maps
are completely independent of coordinates. It follows that, if $\{m_\ell\}$ is a basis for $\Sym_iV$ (or $\bigwedge^i V$) and $\{m_\ell^*\}$ is the corresponding dual basis for $D_i(V^*)$ (or $\bigwedge^i (V^*)$), then the element 
\begin{equation}\label{ev-star}\operatorname{ev}^*(1)=\sum\limits_{\ell} m_\ell^*\otimes m_\ell\in D_i(V^*)\otimes \operatorname{Sym}_iV \quad\text{ (or 
$\textstyle\bigwedge^i (V^*)\otimes 
\bigwedge^i V$)}
\end{equation}
is completely independent of coordinates. These elements will also  be used extensively in our calculations.
\end{chunk}

\begin{subchunk}\label{omega} We emphasize a special case of (\ref{ev-star}). If
  $\omega_{V^*}$ is a basis  for $\bigwedge^{d}(V^*)$ and $\omega_V$ is the corresponding dual basis for $\bigwedge^{d}V$, then the element
$\omega_{V^*}\otimes\omega_V$ is a canonical element of
$\bigwedge^{d}(V^*)\otimes\bigwedge^{d}V$. This element is also used in our calculations.\end{subchunk}

The following   facts about the interaction of the module structures of 
$\bigwedge^\bullet V $ on  $\bigwedge^\bullet(V^*) $ and $\bigwedge^\bullet (V^*) $ on  $\bigwedge^\bullet V $   
 are well known; see \cite[section 1]{BE75}, \cite[Appendix]{BE77},  and \cite[section 1]{Ku93}.

\begin{proposition} \label{A3} Let $V$ be a free module of rank $d$ over a commutative
Noetherian ring $R$ and let $b_r\in \textstyle \bigwedge^{r}V$, $c_p\in \textstyle \bigwedge^{p}V$, and $\alpha_q\in\textstyle \bigwedge^{q}(V^{*})$. 
\begin{enumerate}[\rm(a)]

\item\label{A3.a} If $r =1$, then 
$ \fakeht(b_r(\alpha_q))(c_p)=b_r\wedge (\alpha_q(c_p))+(-1)^{1+q}\alpha_q(b_r\wedge c_p)$. 
\item\label{A3.b} If $q=d$, then
$\fakeht(b_r(\alpha_q))(c_p)=(-1)^{(d -r)(d-p)}(c_p(\alpha_q))(b_r)$. \item\label{A3.c} If $p=d$, then $[b_r(\alpha_q)](c_p)=\fakeht b_r\wedge \alpha_q(c_p)$.
 \item\label{A3.d} If $\Psi:V\to V'$ is a homomorphism of free $R$-modules and $\delta_{s+r}\in \textstyle \bigwedge^{s+r}({V'}^*)$, then 
$$(\textstyle \bigwedge^{s}\Psi^{*})[((\textstyle \bigwedge^{r}\Psi)(b_{r}))(\delta_{s+r})]=
b_{r}[(\textstyle \bigwedge^{s+r}\Psi^{*})(\delta_{s+r})]. \qed$$  \end{enumerate}\end{proposition}

\bigskip
The next result is an application of Proposition~\ref{A3} and the ideas of \ref{87Not1}. The result is used in the proof of Observation~\ref{ann(H0)}.

\begin{proposition}\label{use-in-5.10}   
Let $V$ be a free module of rank $d$ over a commutative
Noetherian ring $R$ and let $b_r\in \textstyle \bigwedge^{r}V$, $c_p\in \textstyle \bigwedge^{p}V$, and $\alpha_q\in\textstyle \bigwedge^{q}(V^{*})$. Then $b_r\wedge (\alpha_q(c_p))$ is equal to a sum of elements of the form
$b'_{r'}\wedge (\alpha'_{q'}(c_p))$
where $$b'_{r'}\in \textstyle \bigwedge^{r'}V,\quad  \alpha'_{q'}\in\textstyle \bigwedge^{q'}(V^{*}),\quad r'-q'=r-q,\quad\text{ and }\quad q'\le q+d-r-p.$$
\end{proposition}
\begin{Remark} The element $c_p$ has not been changed; but an upper bound has been imposed on the degree of the $\bigwedge^\bullet (V^*)$ contribution to the expression. Of course, the assertion is only interesting when $d< r+p$\end{Remark}
\begin{proof}Let $\operatorname{ev}^*(1)=\sum_\ell m_{\ell}^*\otimes m_\ell\in \bigwedge^{d-p}(V^*)\otimes \bigwedge^{d-p}V$, as described in (\ref{ev-star}). Notice that 
\begin{equation}\label{NT}\sum_\ell m_\ell^*(m_\ell\wedge c_p)=c_p.\end{equation} To establish (\ref{NT}), it suffices to test the proposed equation after multiplying both sides on the left by an arbitrary element $x_{d-p}$ of $\bigwedge^{d-p}F$. The left side becomes
\begingroup\allowdisplaybreaks\begin{align*}\sum_\ell x_{d-p}\wedge m_\ell^*(m_\ell\wedge c_p)&=\sum_\ell  [x_{d-p}(m_\ell^*)](m_\ell\wedge c_p)&&\text{by \ref{A3}.\ref{A3.c}}\\
&=\Big(\sum_\ell  [x_{d-p}(m_\ell^*)]\cdot m_\ell\Big)\wedge c_p\\
&=x_{d-p}\wedge c_p,\end{align*}\endgroup as desired. Apply (\ref{NT}) and \ref{A3}.\ref{A3.c} to see that
$$b_r\wedge \alpha_q(c_p)=\sum_\ell b_r\wedge (\alpha_q\wedge m_{\ell}^*)(m_\ell\wedge c_p)=\sum_\ell  [b_r(\alpha_q\wedge m_{\ell}^*)](m_\ell\wedge c_p).$$Each $m_\ell$ is a sum of elements of the form $v_1\wedge \cdots\wedge v_{d-p}$with $v_i\in V$. Apply \ref{A3}.\ref{A3.a} numerous times to write
\begingroup\allowdisplaybreaks\begin{align*}[b_r(\alpha_q \wedge m_\ell^*)](v_1\wedge \cdots\wedge v_{d-p}\wedge c_p)&=\begin{cases}
\pm v_1\wedge [b_r(\alpha_q \wedge m_\ell^*)](v_2\wedge \cdots\wedge v_{d-p}\wedge c_p)\\\pm  [(v_1\wedge b_r)(\alpha_q \wedge m_\ell^*)](v_2\wedge \cdots\wedge v_{d-p}\wedge c_p)\end{cases}\\
&=\begin{cases}
\pm v_2\wedge v_1\wedge [b_r(\alpha_q \wedge m_\ell^*)](v_3\wedge \cdots\wedge v_{d-p}\wedge c_p)
\\
\pm v_1\wedge [(v_2\wedge b_r)(\alpha_q \wedge m_\ell^*)](v_3\wedge \cdots\wedge v_{d-p}\wedge c_p)
\\
\pm v_2\wedge [(v_1\wedge b_r)(\alpha_q \wedge m_\ell^*)](v_3\wedge \cdots\wedge v_{d-p}\wedge c_p)
\\
\pm [(v_2\wedge v_1\wedge b_r)(\alpha_q \wedge m_\ell^*)](v_3\wedge \cdots\wedge v_{d-p}\wedge c_p)
\end{cases}\\&=\cdots \ .\end{align*}\endgroup We continue is this manner and express $b_r\wedge \alpha_q(c_p)$ as a sum of elements of the form 
$$x\wedge[(x'\wedge b_r)(\alpha_q\wedge m_\ell^*)](c_p),$$ where $x$ and $x'$ are homogeneous elements of $\bigwedge^\bullet F$ and $\deg x+\deg x'=d-p$. The proof is complete because 
$$\deg [(x'\wedge b_r)(\alpha_q\wedge m_\ell^*)]=q+d-p-\deg x'-r\le q+d-p-r.$$
\vskip-18pt\end{proof}

\bigskip

\section{The classical generalized Eagon-Northcott complexes.}

\bigskip We recall the classical generalized Eagon-Northcott complexes
$\{\mathcal C^i_\Phi\mid i\in \mathbb Z\}$ which were introduced at (\ref{Phi}).   

\begin{definition}\label{FFR} 
Let $R$ be a commutative Noetherian ring, $F$ and $G$ be free $R$-modules of rank $\goth f$ and $\goth g$, respectively, with $\goth g\le \goth f$, $\Phi:G^*\to F$ be an $R$-module homomorphism,    and $i$  be an integer.
\begin{enumerate}[\rm(a)]
\item The complex 
$\mathcal C_{\Phi}^{i}$
is
\begin{align}\label{all-of-Ci}\textstyle \cdots \xrightarrow{\eta_\Phi} \bigwedge ^{\goth f-\goth g-i-2}F\otimes D_2 (G^*) \xrightarrow{\eta_\Phi} \bigwedge ^{\goth f-\goth g-i-1}F\otimes D_1 (G^*) \xrightarrow{\eta_\Phi} \bigwedge ^{\goth f-\goth g-i}F\otimes D_0(G^*)\vspace{5pt} \\\textstyle
\xrightarrow{\bigwedge^{\goth g}\Phi} \bigwedge ^{\goth f-i}F
\otimes \bigwedge^\goth g G\otimes \Sym_0G \xrightarrow{\Kos_\Phi} 
\bigwedge ^{\goth f-i+1}F
\otimes \bigwedge^\goth g G\otimes \Sym_1 G \xrightarrow{\Kos_\Phi} \cdots,\notag\end{align}
with $\bigwedge ^{\goth f-\goth g-i}F\otimes D_0(G^*)$ in position $i+1$. In particular, if $j$ is an integer, then the module $(\mathcal C^{i}_\Phi)_j$ is 
\begin{equation}\label{Ci-spot}(\mathcal C^{i}_\Phi)_j=\begin{cases}
\bigwedge^{\goth f-j}F\otimes \bigwedge ^{\goth g}G\otimes \Sym_{i-j}G,&\text{if $j\le i$, and}\\ 
\bigwedge^{\goth f-\goth g-j+1}F\otimes D_{j-i-1}(G^*),&\text{if $i+1\le j$.}\end{cases}\end{equation}
\item The maps $\eta_\Phi$ and $\Kos_\Phi$ may be found in Definition~\ref{8mathbb-E}.\ref{8mathbb-E.a}. 
\item The map $\bigwedge^aF\otimes D_0(G^*)\xrightarrow{\bigwedge^\goth g\Phi}\bigwedge^{a+\goth g}F\otimes \bigwedge^\goth gG\otimes \Sym_0G$ is
$$\textstyle f_a\mapsto f_a\wedge (\bigwedge^\goth g \Phi) (\omega_{G^*})\otimes \omega_G,$$ where $\omega_{G^*}\otimes \omega_{G}$ is the canonical element of $\bigwedge^\goth g (G^*)\otimes \bigwedge^\goth g G$ from (\ref{omega}).
\end{enumerate}
\end{definition}

\begin{example}\label{2.1} Adopt the notation of {\rm\ref{FFR}}. 
We record the classical generalized Eagon Northcott complexes $\mathcal C_{\Phi}^{i}$ which have the form:
$$0\to (\mathcal C_{\Phi}^{i})_{\goth f-\goth g+1}\to (\mathcal C_{\Phi}^{i})_{\goth f-\goth g}\to \dots\to (\mathcal C_{\Phi}^{i})_2\to (\mathcal C_{\Phi}^{i})_1
\to (\mathcal C_{\Phi}^{i})_0\to 0.$$Of course, these complexes have length 
$\goth f-\goth g+1$. The corresponding indices $i$ are given in 
(\ref{ClGEN}): 
\begingroup\allowdisplaybreaks
\begin{align}
\textstyle \mathcal C_{\Phi}^{-1}:\ &\textstyle  0\to \bigwedge^0F\otimes D_{\goth f-\goth g+1}(G^*)\xrightarrow{\eta_\Phi}\bigwedge^{1}F\otimes D_{\goth f-\goth g}(G^*)\xrightarrow{\eta_\Phi}\dots \xrightarrow{\eta_\Phi} \bigwedge^{\goth f-\goth g+1}F\otimes D_{0}(G^*)\notag \\ &\to 0,\notag\\
\textstyle \mathcal C_{\Phi}^{0}:\ &\textstyle  0\to \bigwedge^0F\otimes D_{\goth f-\goth g}(G^*)\xrightarrow{\eta_\Phi}\bigwedge^{1}F\otimes D_{\goth f-\goth g-1}(G^*)\xrightarrow{\eta_\Phi}\dots \xrightarrow{\eta_\Phi} \bigwedge^{\goth f-\goth g}F\otimes D_{0}(G^*)
\notag\\&\textstyle\xrightarrow{\bigwedge^\goth g\Phi}   \bigwedge^{\goth f}F\otimes\bigwedge^\goth g G\otimes \Sym_{0}G \to 0,\notag\\
\textstyle \mathcal C_{\Phi}^{1}:\ &\textstyle  0\to \bigwedge^0F\otimes D_{\goth f-\goth g-1}(G^*)\xrightarrow{\eta_\Phi}\bigwedge^{1}F\otimes D_{\goth f-\goth g-2}(G^*)\xrightarrow{\eta_\Phi}\dots \xrightarrow{\eta_\Phi} \bigwedge^{\goth f-\goth g-1}F\otimes D_{0}(G^*)
\notag\\&\textstyle\xrightarrow{\bigwedge^\goth g\Phi}   \bigwedge^{\goth f-1}F\otimes\bigwedge^\goth g G\otimes \Sym_{0}G
\xrightarrow{\Kos_\Phi}\bigwedge^{\goth f}F\otimes\bigwedge^\goth g G\otimes \Sym_{1}G \to 0,\notag\\
\textstyle \mathcal C_{\Phi}^{2}:\ &\textstyle  0\to \bigwedge^0F\otimes D_{\goth f-\goth g-2}(G^*)\xrightarrow{\eta_\Phi}\bigwedge^{1}F\otimes D_{\goth f-\goth g-3}(G^*)\xrightarrow{\eta_\Phi}\dots \xrightarrow{\eta_\Phi} \bigwedge^{\goth f-\goth g-2}F\otimes D_{0}(G^*)
\notag\\&\textstyle\xrightarrow{\bigwedge^\goth g\Phi}   \bigwedge^{\goth f-2}F\otimes\bigwedge^\goth g G\otimes \Sym_{0}G
\xrightarrow{\Kos_\Phi}\bigwedge^{\goth f-1}F\otimes\bigwedge^\goth g G\otimes \Sym_{1}G \notag\\&\textstyle\xrightarrow{\Kos_\Phi}\bigwedge^{\goth f}F\otimes\bigwedge^\goth g G\otimes \Sym_{2}G \to 0,\notag\\
&\hskip2in \vdots\notag\\
\textstyle \mathcal C_{\Phi}^{\goth f-\goth g-1}:\ &\textstyle  0\to \bigwedge^0F\otimes D_{1}(G^*)\xrightarrow{\eta_\Phi}\bigwedge^{1}F\otimes D_{0}(G^*)\notag\\
&\textstyle\xrightarrow{\bigwedge^g\Phi}   \bigwedge^{\goth g+1}F\otimes\bigwedge^\goth g G\otimes \Sym_{0}G
\xrightarrow{\Kos_\Phi}\bigwedge^{\goth f-\goth g-2}F\otimes\bigwedge^\goth g G\otimes \Sym_{1}G \notag\\&\textstyle\xrightarrow{\Kos_\Phi}\dots\xrightarrow{\Kos_\Phi}\bigwedge^{\goth f}F\otimes\bigwedge^\goth g G\otimes \Sym_{\goth f-\goth g-1}G \to 0,\notag\\
\textstyle \mathcal C_{\Phi}^{\goth f-\goth g}:\ &\textstyle  0\to \bigwedge^0(F^*)\otimes D_{0}(G^*)
\notag\\
&\textstyle\xrightarrow{\bigwedge^\goth g\Phi}   \bigwedge^{\goth g}F\otimes\bigwedge^\goth g G\otimes \Sym_{0}G
\xrightarrow{\Kos_\Phi}\bigwedge^{\goth g+1}F\otimes\bigwedge^\goth g G\otimes \Sym_{1}G \notag\\&\textstyle\xrightarrow{\Kos_\Phi}\dots\xrightarrow{\Kos_\Phi}\bigwedge^{\goth f}F\otimes\bigwedge^\goth g G\otimes \Sym_{\goth f-\goth g}G \to 0,\text{ and}\notag\\
\textstyle \mathcal C_{\Phi}^{\goth f-\goth g+1}:\ &\textstyle  0\to \bigwedge^{\goth g-1}F\otimes\bigwedge^\goth g G\otimes  \Sym_{0}G
\xrightarrow{\Kos_\Phi}   \bigwedge^{\goth g}F\otimes\bigwedge^\goth g G\otimes \Sym_{1}G\notag\\&\textstyle
\xrightarrow{\Kos_\Phi}\dots\xrightarrow{\Kos_\Phi}\bigwedge^{\goth f}F\otimes\bigwedge^\goth g G\otimes \Sym_{\goth f-\goth g+1}G\to 0. \notag  
\end{align}\endgroup
The maps $\eta_\Phi$ and $\Kos_\Phi$ may be found in Definition~\ref{8mathbb-E}.\ref{8mathbb-E.a}.
\end{example}

\bigskip

\section{Schur and Weyl modules which correspond to hooks.}

\bigskip
The modules which comprise the complexes $\{\mathcal C^{i,a}_\Phi\}$ are Schur and Weyl modules which correspond to hooks. We recall some of the elementary properties of these modules. 
\begin{chunk}\label{2.3}
Let $V$ be a non-zero free module of rank $d$ over the commutative Noetherian ring $R$ and let $a$ and $b$ be integers. Define the $R$-module homomorphisms 
$${\alignedat1&\kappa_{b}^{a }:\  \textstyle{\bigwedge} ^aV\otimes  \operatorname {Sym} _bV\to \textstyle{\bigwedge} ^{a-1}V\otimes  \operatorname {Sym} _{b+1}V
\quad \text{and}\\\vspace{5pt} &\eta_{b}^{a }:\ \textstyle{\bigwedge} ^aV\otimes  D _bV\to \textstyle{\bigwedge} ^{a+1}V\otimes  D _{b-1}V\endalignedat}$$ to be the compositions
\begin{align}\label{co-m1}\textstyle\textstyle{\bigwedge} ^aV\otimes  \operatorname {Sym} _bV\xrightarrow{\ev^*(1)}&\textstyle
(V^*\otimes V)\otimes {\bigwedge} ^{a}V\otimes \operatorname {Sym} _{b}V\\\textstyle\xrightarrow{\text{rearrange}}&\textstyle
(V^*\otimes {\bigwedge} ^{a}V)\otimes (V\otimes \operatorname {Sym} _{b}V)\notag\\\textstyle
\xrightarrow{\modact\otimes \text{mult}}& \textstyle{\bigwedge} ^{a-1}V\otimes  \operatorname {Sym} _{b+1}V &\textstyle\quad\text{and}\notag\\
\label{co-m2} \textstyle{\bigwedge} ^aV\otimes  D _bV\xrightarrow{\ev^*(1)} &\textstyle{\bigwedge} ^{a}V\otimes (V\otimes V^*)\otimes D _{b}V\\\textstyle\notag
\xrightarrow{\text{regroup}} &
(\textstyle{\bigwedge} ^{a}V\otimes V)\otimes (V^*\otimes D _{b}V)\\\textstyle\notag
\xrightarrow{\mult\otimes\modact}& 
\textstyle{\bigwedge} ^{a+1}V\otimes  D _{b-1}V,&
\end{align} respectively. The map $\ev^*(1)$ is discussed in (\ref{ev-star}), ``$\mult$'' is multiplication in the symmetric algebra or the exterior algebra, and ``$\modact$'' is the module action of $\bigwedge^\bullet (V^*)$ on $\bigwedge^\bullet V$ or $\Sym_\bullet (V^*)$ on $D_\bullet V$. Define the $R$-modules
$$L^{a }_{b}V=\ker \kappa_{b}^{a }\quad \text{and}\quad  K^{a }_{b}V=\ker \eta_{b}^{a }.$$ 
In the future, we will often write 
$\kappa$ and $\eta$ in place of $\kappa_{b}^{a }$ and $\eta_{b}^{a }$. 
\end{chunk}

\begin{chunk}
The $R$-modules $L^{a }_{b}V$ and $K^{a }_{b}V$  have been used by many authors in many contexts. In particular, they are studied extensively in \cite{BE75}; although our indexing conventions are different than the conventions of \cite{BE75}; that is,
$$\text{the module we call $L^{a }_{b}V$ is called ${L_b}^{a+1}V$ in \cite{BE75}.}$$\end{chunk}

\begin{chunk}\label{4.6}The modules $L^{a }_{b}V$ and $K^{a }_{b} V$ may also be thought of as the Schur modules $L_\lambda V$ and Weyl modules $K_\lambda V$ which correspond to certain hooks $\lambda$. We use the notation of Examples 2.1.3.h and 2.1.17.h in Weyman \cite{W} to see that 
the module we call $L^{a }_{b}V$ is also the Schur module $L_{(a+1,1^{b-1})}V$ and the module we call  $K^{a }_b V$ is also the  Weyl module $K _{(b+1,1^{a-1})} V$. 
\end{chunk}

\begin{chunk}The complex 

\begin{align}\label{KOS} 0\to \textstyle{\bigwedge} ^dV\otimes  \operatorname {Sym} _{c-d}V\xrightarrow{\kappa} \textstyle{\bigwedge} ^{d-1}V\otimes  \operatorname {Sym} _{c-d+1}V\xrightarrow{\kappa}\\\vspace{5pt}  
\cdots \xrightarrow{\kappa} \textstyle{\bigwedge} ^1V\otimes  \operatorname {Sym} _{c-1}V\xrightarrow{\kappa}\textstyle{\bigwedge} ^0V\otimes  \operatorname {Sym} _{c}V\to 0,&\notag\end{align}
 which is a homogeneous strand of an acyclic  Koszul complex, is split exact for all non-zero integers $c$; hence, $L^{a }_{b}V$ is a projective   $R$-module. In fact, $L^{a }_{b}V$ is a free $R$-module of rank 
$$\operatorname{rank}  L^{a }_{b}V=\binom{d+b-1}{a+b}\binom{a+b-1}a;$$see \cite[Prop. 2.5]{BE75}.
Similarly, $K^a_{b}V$ is a  free $R$-module of rank
$$\operatorname{rank} K^a_{b}V=\binom{d+b}{a+b}\binom{a+b-1}{b}.
$$
The perfect pairing 
$$\left(\vphantom{E^{E^E}}{\textstyle \bigwedge} ^aV\otimes \operatorname {Sym}_b V\right)\otimes  \left(\vphantom{E^{E^E}}{\textstyle \bigwedge} ^{a}(V^*)\otimes D_b (V^*)\right)\to R,$$
induces a perfect pairing 
\begin{equation}\label{p207}L^{a }_{b}V\otimes K^{a+1 }_{b-1}(V^*) \to R,\quad\text{provided $(a,b)\neq (0,0)$ or $(-1,1)$}.\end{equation}
\end{chunk}

\bigskip
Each assertion in Observation~\ref{2.8} is obvious; but it is very convenient to have all of these facts gathered in one place.

\begin{observation}\label{2.8}Let $R$ be a commutative Noetherian ring, $V$ be a free $R$-module of rank $d$, and $\ell$ be an integer. Then the following modules are canonically isomorphic{\rm:}
\begin{enumerate}[\rm(a)]
\item\label{2.8.c} $L^0_\ell V=\Sym_\ell V$,
\item\label{2.8.b} $L_{\ell}^{d-1} V\cong\bigwedge^{d}  V\otimes \Sym_{\ell-1}  V$, provided $\ell+d\neq 1$,
\item\label{2.8.e} $L^d_\ell V=0$, provided $\ell+d\neq 0$,
\item\label{2.8.i} $L_0^\ell V=0$, provided $\ell\neq 0$,
\item\label{2.8.g} $L_1^\ell V\cong\bigwedge^{\ell+1} V$, provided $\ell\neq -1$,
\item\label{2.8.f} $K^0_\ell V=0$, provided $\ell\neq 0$,
\item\label{2.8.a} $K_{\ell}^1 V\cong D_{\ell+1} V$, provided $\ell\neq -1$,
\item\label{2.8.d} $K^d_\ell V=\bigwedge^d V\otimes D_\ell V$, and
\item\label{2.8.h} $K_0^\ell V=\bigwedge^{\ell} V$.
\end{enumerate}
\end{observation}
\begin{proof} Assertions (\ref{2.8.c}), (\ref{2.8.d}), and (\ref{2.8.h}) follow from the definitions,
and assertions (\ref{2.8.b}), (\ref{2.8.e}), (\ref{2.8.i}), and  (\ref{2.8.g})  are immediate consequences of 
split exact complex  (\ref{KOS}). 
 If $c$ is non-zero, then we may apply $\Hom(-,R)$ to the split exact complex  (\ref{KOS}) to obtain the split exact complex 
\begin{equation}\notag\textstyle 0\to \bigwedge^0(V^*)\otimes D_c(V^*)\xrightarrow{\eta}  
\bigwedge^1(V^*)\otimes D_{c-1}(V^*)\xrightarrow{\eta}
\dots\xrightarrow{\eta} \bigwedge^d(V^*)\otimes D_{c-d}(V^*)\to 0.\end{equation}
Replace $V^*$ with $V$ in order to see that 
\begin{equation}\label{KOS*}\textstyle 0\to \bigwedge^0V\otimes D_cV\xrightarrow{\eta}  
\bigwedge^1V\otimes D_{c-1}V\xrightarrow{\eta}
\dots\xrightarrow{\eta} \bigwedge^dV\otimes D_{c-d}V\to 0\end{equation}
is also split exact.
Assertions~(\ref{2.8.f}) and (\ref{2.8.a}) are  consequence of (\ref{KOS*}). \end{proof}

\bigskip

\section{The complexes 
$\mathbb K_\Phi$, and $\mathbb L_\Phi$ associated to a homomorphism $\Phi$.}\label{K-and-L}

\bigskip 
The complexes $\mathbb K$ and $\mathbb L$ contain Weyl modules $K^a_b(G^*)$ and Schur modules $L^a_bG$, respectively. The complex $\mathcal C^{i,a}_{\Phi}$, which is the focal point of the present paper, is obtained by concatenating the complexes $\mathbb K\to \bigwedge \to \mathbb L$; see Definition~\ref{CiaPhi}.

\begin{data}\label{8'setup-modified}Let $R$ be a commutative Noetherian ring, $F$ and $G$ be free  $R$-modules of finite rank $\goth f$ and $\goth g$, respectively,
and $\Phi:G^*\to F$ be an $R$-module homomorphism. 
\end{data}

\begin{definition}\label{8mathbb-E}
Adopt Data~\ref{8'setup-modified}. 
\begin{enumerate}[\rm(a)]\item\label{8mathbb-E.a} If $r$ and $q$ are integers, then define the $R$-module homomorphisms 
\begin{align*}\textstyle\eta_\Phi: &\textstyle\bigwedge^rF\otimes D_q(G^*)\to \bigwedge^{r+1}F\otimes D_{q-1}(G^*)&&\text{and}\\ 
\Kos_\Phi: &\textstyle\bigwedge^rF\otimes \Sym_qG\to \bigwedge^{r+1}F\otimes \Sym_{q+1}G
\end{align*} to be the compositions
\begin{align*}\textstyle
\bigwedge^rF\otimes  D_q(G^*)\xrightarrow{1\otimes \ev^*(1)\otimes 1}&\textstyle\bigwedge^rF\otimes G^*\otimes G\otimes  D_q(G^*)\\\xrightarrow{1\otimes \Phi\otimes \modact} &\textstyle\bigwedge^{r}F\otimes \bigwedge^{1}F\otimes  D_{q-1}(G^*)\xrightarrow{\mult\otimes 1}\bigwedge^{r+1}F\otimes  D_{q-1}(G^*)\\\intertext{and}
\textstyle\bigwedge^rF\otimes  \Sym_qG\xrightarrow{1\otimes \ev^*(1)\otimes 1}&\textstyle\bigwedge^rF\otimes G^*\otimes G\otimes  \Sym_qG\\\xrightarrow{1\otimes \Phi\otimes \mult} &\textstyle\bigwedge^{r}F\otimes \bigwedge^{1}F\otimes  \Sym_{q+1}G\xrightarrow{\mult\otimes 1}\bigwedge^{r+1}F\otimes  \Sym_{q+1}G,\end{align*} 
respectively, where $\ev^*(1)$ is described in {\rm(\ref{ev-star})}, $\modact$ is the module action of $\Sym_\bullet G$ on $D_\bullet(G^*)$ and $\mult$ is multiplication in the exterior algebra $\bigwedge^{\bullet} F$ or the symmetric algebra $\Sym_{\bullet}G$.

\item\label{8mathbb-E.c} If $N$ and $p$ are integers, then define  $\mathbb K^{N,p}_{\Phi}$ to be the maps and modules 
$$\textstyle \mathbb K^{N,p}_{\Phi}:\textstyle \  0\to \bigwedge^0F\otimes K^p_N(G^*)\xrightarrow{\eta_\Phi} \bigwedge^1F\otimes K^{p}_{N-1}(G^*) \xrightarrow{\eta_\Phi}  \dots  \xrightarrow{\eta_\Phi} \bigwedge^NF\otimes K^{p}_{0}(G^*)\to 0$$
with  $[\mathbb K^{N,p}_{\Phi}]_j=\bigwedge^{N-j}F\otimes K^{p}_{j}(G^*)$; and 
define  $\mathbb L^{N,p}_{\Phi}$ to be the maps and modules 
$$
\textstyle   0\to \bigwedge^{N+1}F\otimes L_1^pG \xrightarrow{\Kos_\Phi}
\bigwedge^{N+2}F\otimes L_2^pG\xrightarrow{\Kos_\Phi}
\cdots \xrightarrow{\Kos_\Phi}\bigwedge^{\goth f}F\otimes L_{\goth f-N}^pG\to 0,
$$ 
with \begin{equation}\label{cases}[\mathbb L^{N,p}_{\Phi}]_j=\begin{cases}0&\text{if $(p,j)=(0,\goth f-N)$, and}\\\bigwedge^{\goth f-j}F\otimes L_{\goth f-N-j}^pG,&\text{otherwise,}\end{cases}\end{equation} 
where the maps are induced by the homomorphisms $\eta_\Phi$ and $\Kos_\Phi$ of (\ref{8mathbb-E.a}) and the modules $K^p_q(G^*)$ and $L^p_qG$ are defined in (\ref{co-m2}) and (\ref{co-m1}).
\end{enumerate}
\end{definition}

\begin{remarks}\label{5.3}\begin{enumerate}[\rm(a)] \item The map $$\textstyle\eta^a_b:\bigwedge^a(G^*)\otimes D_b(G^*)\to  \bigwedge^{a+1}(G^*)\otimes D_{b-1}(G^*)$$ of (\ref{co-m2}) is equal to the map $$\textstyle\eta_{\id_{G^*}}:\bigwedge^a(G^*)\otimes D_b(G^*)\to  \bigwedge^{a+1}(G^*)\otimes D_{b-1}(G^*)$$ of \ref{8mathbb-E}.\ref{8mathbb-E.a}. \item The map $$\textstyle\kappa^a_b: \bigwedge^aG\otimes \Sym_aG\to \bigwedge^{a-1}G\otimes \Sym_{b+1}G$$ of (\ref{co-m1}) is related to the map $\Kos_{\id_{G^*}}$ of \ref{8mathbb-E}.\ref{8mathbb-E.a} by way of the following commutative diagram:
$$\xymatrix{\bigwedge^{\goth g-a}G^*\otimes \Sym_bG\otimes \bigwedge ^\goth g G\ar[d]_{\modact}^{\cong}\ar[rrr]^{(-1)^{\goth g-a}\Kos_{\id_{G^*}}}&&& \bigwedge^{\goth g-a+1}G^*\otimes \Sym_{b+1}G\otimes \bigwedge ^\goth g G\ar[d]^{\modact}_{\cong}\\
\bigwedge^aG\otimes \Sym_bG\ar[rrr]^{\kappa^a_b}&&& \bigwedge^{a-1}G\otimes \Sym_{b+1}G.}
$$

\item\label{5.3.c} The maps and modules of $\mathbb K^{N,p}_\Phi$ and $\mathbb L^{N,p}_\Phi$ form complexes because the diagrams
{\footnotesize{$$
\xymatrix{0\to\bigwedge^rF\otimes K^p_q(G^*)\ar@{^(->}[r]\ar@{-->}[d]^{\eta_\Phi}&\bigwedge^rF\otimes \bigwedge^p(G^*)\otimes D_q(G^*)\ar[r]^(.48){\eta_{\id_{G^*}}}\ar[d]^{\eta_\Phi}&
\bigwedge^rF\otimes \bigwedge^{p+1}(G^*)\otimes D_{q-1}(G^*)\ar[d]^{\eta_\Phi}\\
0\to\bigwedge^{r+1}F\otimes K^p_{q-1}(G^*)\ar@{^(->}[r]&\bigwedge^{r+1}F\otimes \bigwedge^p(G^*)\otimes D_{q-1}(G^*)\ar[r]^(.48){\eta_{\id_{G^*}}}&
\bigwedge^{r+1}F\otimes \bigwedge^{p+1}(G^*)\otimes D_{q-2}(G^*)
}$$}}

\noindent and
{\footnotesize{$$\xymatrix{0\to\bigwedge^rF\otimes L^p_qG\ar@{^(->}[r]\ar@{-->}[d]^{\Kos_\Phi}&\bigwedge^rF\otimes \bigwedge^pG\otimes \Sym_qG\ar[r]^(.48){\kappa}\ar[d]^{\Kos_\Phi}&
\bigwedge^rF\otimes \bigwedge^{p-1}G\otimes \Sym_{q+1}G\ar[d]^{\Kos_\Phi}\\
0\to\bigwedge^{r+1}F\otimes L^p_{q+1}G\ar@{^(->}[r]&\bigwedge^{r+1}F\otimes \bigwedge^pG\otimes \Sym_{q+1}G\ar[r]^(.48){\kappa}&
\bigwedge^{r+1}F\otimes \bigwedge^{p-1}G\otimes \Sym_{q+2}G}
$$}}

\noindent commute.
\end{enumerate}\end{remarks}

\bigskip 

\section{The complexes $\mathbb K_\Phi$ and $\mathbb L_\Phi$ when $\Phi$ is a direct sum of homomorphisms.}

\bigskip 

In this section we assume that the homomorphism $\Phi:G^*\to F$ is a direct sum of homomorphisms. In Proposition~\ref{87.5} we relate the complexes $\mathbb K^{N,p}_\Phi$ and $\mathbb L^{N,p}_\Phi$ of Definition~\ref{8mathbb-E}.\ref{8mathbb-E.c} to similar complexes built from smaller data. This result plays a prominent role in the proof of the acyclicity of the complexes $\mathcal C^{i,a}_\Phi$; see Theorem~\ref{main}, which is the main result of the paper. 
\begin{data}\label{8data-8-8} 
Let $R$ be a commutative Noetherian ring, $F$ and $G$ be free $R$-modules of rank $\goth f$ and $\goth g$, respectively,  
and $\Phi:G^*\to F$ be an $R$-module homomorphism. 
Decompose $F$ and $G$ as
\begin{equation}\notag
F=F'\oplus F'' \quad\text{and}\quad G={G'}\oplus {G''},\end{equation} where $F'$, $F''$, ${G'}$ and ${G''}$  are free $R$-modules and $\rank F''=\rank {G''}=1$ and  let $$F^*={F'}^*\oplus {F''}^*\quad\text{and}\quad  G^*={G'}^*\oplus {G''}^*$$ be the corresponding decompositions of $F^*$ and $G^*$.
Assume that
 \begin{equation}\notag\Phi=\bmatrix \Phi'&0\\0&\Phi''\endbmatrix,\end{equation} where $\Phi':{G'}^*\to F'$ 
and 
$\Phi'':{G''}^*\to F''$ are  $R$-module homomorphisms.
\end{data}

\begin{notation}\label{notation-9}Adopt Data~\ref{8data-8-8}.\begin{enumerate}[\rm(a)]\item
Let $\mathbb A(\Phi'')$ and $\mathbb B(\Phi'')$ represent the complexes \begin{align*}\mathbb A(\Phi''):\quad &0\to R\xrightarrow{\Kos_{\Phi''} }F''\otimes G''\to 0\\\intertext{and}
\mathbb B(\Phi''):\quad &0\to {G''}^*\xrightarrow{\Phi''}F'' \to 0.
\end{align*}
\item Let $\pi':F\to F'$ be the projection map  
which corresponds to the direct sum decomposition $F=F'\oplus F''$.

\item\label{notation-9.c} If $a$, $b$, $c$, $d$, and $e$ are integers, then let $$\textstyle\incl^\dagger:\bigwedge^aF'\otimes \bigwedge^bG\otimes \Sym_cG\otimes\bigwedge^dF''\otimes \Sym_eG''\to \bigwedge^{a+d}F\otimes \bigwedge^bG\otimes \Sym_{c+e}G$$ be the $R$-module homomorphism given by
$$\textstyle (\text{multiplication in} \textstyle\bigwedge^\bullet F)\otimes(\text{the identity map in} \bigwedge^\bullet G)\otimes(\text{multiplication in} \Sym_\bullet G).$$ 
\item\label{notation-9.d} If $a$, $b$, $c$, $d$, and $e$  are integers, then let $$\textstyle\quot^\dagger: \bigwedge^{a+d}F\otimes \bigwedge^{b+e}G\otimes \Sym_{c}G\to \bigwedge^aF'\otimes \bigwedge^bG'\otimes \Sym_cG'\otimes\bigwedge^dF''\otimes \bigwedge^eG''$$  be the $R$-module homomorphism given by
{\small$$\textstyle \left(\begin{matrix}\text{the projection map}\\\text{$\bigwedge^{a+d}F\to
 \bigwedge^aF'\otimes \bigwedge^d F''$}\\\text{induced by $F=F'\oplus F''$}\end{matrix} \right)\otimes
\left(\begin{matrix}\text{the projection map}\\\text{$\bigwedge^{b+e}G\to
 \bigwedge^bG'\otimes \bigwedge^eG''$}\\\text{induced by $G=G'\oplus G''$}\end{matrix} \right)\otimes \left(\begin{matrix}\text{the  quotient map}\\\text{$\Sym_\bullet G\to \Sym_\bullet G'$}\\\text{induced by $G/G''=G'$}\end{matrix}\right).$$}
\item\label{notation-9.e} If $a$, $b$, $c$,  and $d$ are integers, then let $$\textstyle\incl^\ddag:\bigwedge^aF'\otimes \bigwedge^b({G'}^*)\otimes D_c({G'}^*)\otimes\bigwedge^d({G''}^*)\to \bigwedge^{a}F\otimes \bigwedge^{b+d}(G^*)\otimes D_{c}(G^*)$$ be the $R$-module homomorphism given by
$$(\text{inclusion})\otimes(\text{multiplication})\otimes(\text{inclusion}).$$
\item\label{notation-9.f} If $a$, $b$, $c$, $d$, and $e$  are integers, then let $$\textstyle\quot^\ddag: \bigwedge^{a+d}F\otimes \bigwedge^{b}(G^*)\otimes D_{c+e}(G^*)\to \bigwedge^aF'\otimes \bigwedge^b({G'}^*)\otimes D_c(G^*)\otimes\bigwedge^dF''\otimes D_e({G''}^*)$$  be the $R$-module homomorphism given by
{\small$$\textstyle \left(\begin{matrix}\text{the projection map}\\\text{$\bigwedge^{a+d}F\to
 \bigwedge^aF'\otimes \bigwedge^d F''$}\\\text{induced by $F=F'\oplus F''$}\end{matrix} \right)\otimes
\left(\begin{matrix}\text{the identity  map}\\\text{on $\bigwedge^{\bullet}(G^*)$}\end{matrix} \right)\otimes \left(\begin{array}{l} D_{c+e}(G^*)\xrightarrow{\ev^*(1)\otimes 1}\vspace{5pt}\\  D_e({G''}^*)\otimes \Sym_eG''\otimes D_{c+e}(G^*)\vspace{5pt}
\\ \xrightarrow{1\otimes \modact} D_e({G''}^*)\otimes D_c(G^*)\vspace{5pt}\\
\xrightarrow{\text{exchange}}D_c(G^*)\otimes D_e({G''}^*)
\end{array}\right),$$}where $\ev^*(1)$ is described in (\ref{ev-star}) and $\modact$ is the module action of $\Sym_\bullet G$ on $D_\bullet (G^*)$.\end{enumerate}
\end{notation}

Proposition~\ref{87.5} asserts that (\ref{87.5.gtsL}) and (\ref{87.5.gts}) are short exact sequences of complexes. Observation~\ref{8-21-0} considers the modules in (\ref{87.5.gtsL}) and (\ref{87.5.gts}) in the arbitrary position $j$.
Recall the function $\chi$ from \ref{chi}.
\begin{observation}\label{8-21-0} Adopt Data~{\rm\ref{8data-8-8}} and Notation~{\rm\ref{notation-9}}. 
\begin{enumerate}[\rm(a)]
\item\label{8-21-0.a}If $p,q,r$ are integers, then
{\footnotesize\begin{equation}\notag 0\to \begin{matrix} \bigwedge^rF'\otimes L^p_qG\\\oplus\\
\bigwedge^{r-1}F'\otimes L^p_{q-1}G\otimes F''\otimes G''\end{matrix}
\xrightarrow{\incl^\dagger} {\textstyle\bigwedge^r}F\otimes L^p_qG \xrightarrow{\quot^\dagger} 
\begin{matrix} \bigwedge^{r-1}F'\otimes L^p_qG'\otimes F''\\\oplus\\
\chi((p,q)\neq (1,0))\big(\bigwedge^{r-1}F'\otimes L^{p-1}_{q}G'\otimes F''\otimes G''\big)\end{matrix}\to 0\end{equation}}

\noindent is an exact sequence of $R$-modules.

\item \label{8-21-0.b} If $p,q,r$ are integers, 
 then
{\footnotesize\begin{equation}\notag
\textstyle 0\to{\begin{matrix} \bigwedge^{r}F'\otimes K^p_q({G'}^*)\\\oplus\\
\bigwedge^{r}F'\otimes K^{p-1}_{q}({G'}^*)\otimes {G''}^*\end{matrix}}
\xrightarrow{\incl^\ddag} \bigwedge^{r}F\otimes K^p_q(G^*)\xrightarrow{\quot^\ddag}
{\begin{matrix} 
\chi\big((p,q)\neq(0,1)\big)\big(\bigwedge^{r}F'\otimes K^p_{q-1}(G^*)\otimes {G''}^*\big)\\\oplus\\
\bigwedge^{r-1}F'\otimes K^{p}_{q}(G^*)\otimes F''\end{matrix}}
\to0\end{equation}}

\noindent is 
 an exact sequence of $R$-modules.\end{enumerate}\end{observation}

\setcounter{figure}{\value{equation}}
\begin{figure}
\begin{center}
{\tiny$$ \xymatrix{0\ar[d]&0\ar[d]&0\ar[d]
\\
0\to{\begin{matrix} \bigwedge^{r}F'\otimes L^p_qG\\\oplus\\
\bigwedge^{r-1}F'\otimes L^p_{q-1}G\otimes F''\otimes G''\end{matrix}}
\ar[r]^(.6){\incl^\dagger}\ar@{^(->}[d]& \bigwedge^{r}F\otimes L^p_qG\ar[r]^(.37){\quot^\dagger}\ar@{^(->}[d]&
{\begin{matrix} 
\bigwedge^{r-1}F'\otimes L^p_qG'\otimes F''\\\oplus\\
\bigwedge^{r-1}F'\otimes L^{p-1}_{q}G'\otimes F''\otimes G''\end{matrix}}
\ar@{^(->}[d]\to0
\\
0\to{\begin{matrix} \bigwedge^{r}F'\otimes \bigwedge^pG\otimes S_qG\\\oplus\\
\bigwedge^{r-1}F'\otimes \bigwedge^pG\otimes S_{q-1}G\otimes F''\otimes G''\end{matrix}}
\ar[r]^(.6){\incl^\dagger}\ar[d]^{\kappa}& \bigwedge^{r}F\otimes \bigwedge^pG\otimes S_qG\ar[r]^(.37){\quot^\dagger}\ar[d]^{\kappa}&
{\begin{matrix} 
\bigwedge^{r-1}F'\otimes \bigwedge^pG'\otimes S_qG'\otimes F''\\\oplus\\
\bigwedge^{r-1}F'\otimes \bigwedge^{p-1}G'\otimes S_{q}G'\otimes F''\otimes G''\end{matrix}}
\ar[d]^{\kappa}\to0
\\
0\to{\begin{matrix} \bigwedge^{r}F'\otimes \bigwedge^{p-1}G\otimes S_{q+1}G\\
\oplus\\
\bigwedge^{r-1}F'\otimes \bigwedge^{p-1}G\otimes S_{q}G\otimes F''\otimes 
G''\end{matrix}}
\ar[r]^(.6){\incl^\dagger}\ar[d]^{\kappa}& \bigwedge^{r}F\otimes \bigwedge^{p-1}G\otimes S_{q+1}G\ar[r]^
(.37){\quot^\dagger}\ar[d]^{\kappa}&
{\begin{matrix} 
\bigwedge^{r-1}F'\otimes \bigwedge^{p-1}G'\otimes S_{q+1}G'\otimes F''\\\oplus\\
\bigwedge^{r-1}F'\otimes \bigwedge^{p-2}G'\otimes S_{q+1}G'\otimes F''\otimes 
G''\end{matrix}}
\ar[d]^{\kappa}\to0
\\
0\to{\begin{matrix} \bigwedge^{r}F'\otimes L^{p-2}_{q+2}G\\\oplus\\
\bigwedge^{r-1}F'\otimes L^{p-2}_{q+1}G\otimes F''\otimes G''\end{matrix}}
\ar[r]^(.6){\incl^\dagger}\ar[d]& \bigwedge^{r}F\otimes L^{p-2}_{q+2}G\ar[r]^(.37){\quot^\dagger}\ar[d]&
{\begin{matrix} 
\bigwedge^{r-1}F'\otimes L^{p-2}_{q+2}G'\otimes F''\\\oplus\\
\bigwedge^{r-1}F'\otimes L^{p-3}_{q+2}G'\otimes F''\otimes G''\end{matrix}}
\to0\ar[d]\\
0&0&0}$$}\caption{This picture is a commutative diagram which is used in the proof of Observation~\ref{8-21-0}.\ref{8-21-0.a}. The middle two rows are exact. The columns are exact provided $2\le p+q$. We use ``$S$'' as an abbreviation for ``$\Sym$''.}\setcounter{figure}{\value{equation}}\label{the following picture}\addtocounter{equation}{1}
\end{center}\end{figure}

\setcounter{figure}{\value{equation}}
\begin{figure}
\begin{center}
{\tiny$$\hskip-1.6pt \xymatrix{0\ar[d]&0\ar[d]&0\ar[d]
\\
0\to{\begin{matrix} \bigwedge^{r}F'\otimes K^p_q({G'}^*)\\\oplus\\
\bigwedge^{r}F'\otimes K^{p-1}_{q}({G'}^*)\otimes {G''}^*\end{matrix}}
\ar[r]^(.6){\incl^\ddag}\ar@{^(->}[d]& \bigwedge^{r}F\otimes K^p_q(G^*)\ar[r]^(.42){\quot^\ddag}\ar@{^(->}[d]&
{\begin{matrix} 
\bigwedge^{r}F'\otimes K^p_{q-1}(G^*)\otimes {G''}^*\\\oplus\\
\bigwedge^{r-1}F'\otimes K^{p}_{q}(G^*)\otimes F''\end{matrix}}
\ar@{^(->}[d]\to0\\
0\to{\begin{matrix} \bigwedge^{r}F'\otimes \bigwedge^p({G'}^*)\otimes D_q({G'}
^*)\\\oplus\\
\bigwedge^{r}F'\otimes \bigwedge^{p-1}({G'}^*)\otimes D_{q}({G'}^*)\otimes  {G''}
^*
\end{matrix}}
\ar[r]^(.6){\incl^\ddag}\ar[d]^{\eta_{\id_{G^*}}}& \bigwedge^{r}F\otimes \bigwedge^p(G^*)\otimes D_q(G^*)\ar[r]^(.42){\quot^\ddag}\ar[d]^{\eta_{\id_
{G^*}}}&
{\begin{matrix} 
\bigwedge^{r}F'\otimes \bigwedge^p(G^*)\otimes D_{q-1}(G^*)\otimes {G''}^*\\
\oplus\\
\bigwedge^{r-1}F'\otimes \bigwedge^{p}(G^*)\otimes D_{q}(G^*)\otimes F''\end
{matrix}}
\ar[d]^{\eta_{\id_{G^*}}}\to0
\\
0\to\hskip-2pt{\begin{matrix} \bigwedge^{r}F'\otimes \bigwedge^{p+1}({G'}^*)\otimes D_{q-1}({G'}^*)\\
\oplus\\
\bigwedge^{r}F'\otimes \bigwedge^{p}({G'}^*)\otimes D_{q-1}({G'}^*)\otimes 
{G''}^*\end{matrix}}
\ar[r]^(.6){\incl^\ddag}\ar[d]^{\eta_{\id_{G^*}}}&\bigwedge^{r}F\otimes \bigwedge^{p+1}(G^*)\otimes D_{q-1}(G^*)\ar[r]^
(.42){\quot^\ddag}\ar[d]^{\eta_{\id_{G^*}}}&
{\begin{matrix} 
\bigwedge^{r}F'\otimes \bigwedge^{p+1}(G^*)\otimes D_{q-2}(G^*) \otimes {G''}^*
\\
 \oplus\\
\bigwedge^{r-1}F'\otimes \bigwedge^{p+1}(G^*)\otimes D_{q-1}(G^*)\otimes F''
\end{matrix}}
\ar[d]^{\eta_{\id_{G^*}}}\to0\\
0\to{\begin{matrix} \bigwedge^{r}F'\otimes K^{p+2}_{q-2}({G'}^*)\\\oplus\\
\bigwedge^{r}F'\otimes K^{p+1}_{q-2}({G'}^*)\otimes {G''}^*\end{matrix}}
\ar[r]^(.6){\incl^\ddag}\ar[d]& \bigwedge^{r}F\otimes K^{p+2}_{q-2}(G^*)\ar[r]^(.42){\quot^\ddag}\ar[d]&
{\begin{matrix} 
\bigwedge^{r}F'\otimes K^{p+2}_{q-3}(G^*)\otimes {G''}^*\\\oplus\\
\bigwedge^{r-1}F'\otimes K^{p+2}_{q-2}(G^*)\otimes F''\end{matrix}}
\to0\ar[d]\\
0&0&0}$$}\caption{This picture is a commutative diagram which is used in the proof of Observation~\ref{8-21-0}.\ref{8-21-0.b}. The middle two rows are exact. The columns are exact provided $2\le p+q$.}\setcounter{figure}{\value{equation}}\label{not the following picture}\addtocounter{equation}{1}
\end{center}\end{figure}

\begin{remark-no-advance} 
If $r=1$ and $F'=0$, then $F=F''$ has rank one and no harm occurs if we set $F=F''$ equal to $R$ (that is, apply $-\otimes F^*$). In this case,
\ref{8-21-0}.\ref{8-21-0.a} is
\begin{equation}\label{dentist}0\to 
L^p_{q-1}G\otimes G''
\xrightarrow{\incl^\dagger}  L^p_qG \xrightarrow{\quot^\dagger} 
 L^p_qG'\oplus
\chi\big((p,q)\neq(1,0)\big)\big(L^{p-1}_{q}G'\otimes G''\big)\to 0.\end{equation}
If $r=0$, then \ref{8-21-0}.\ref{8-21-0.b} is 
\small{\begin{equation}\label{rep-thy}\textstyle 0\to \begin{matrix}
K^p_q({G'}^*)\\
\oplus \\
 \big(K^{p-1}_q({G'}^*)\otimes {G''}^*\big)\end{matrix}
\xrightarrow{\incl^\ddag} K^p_q(G^*)
\xrightarrow{\quot^\ddag} 
\chi\big((p,q)\neq(0,1)\big)\big(K^p_{q-1}(G^*)\otimes {G''}^*\big)\to 0.\end{equation} }
The exact sequences (\ref{dentist}) and (\ref{rep-thy}) are  results from representation theory; see, for example,  \cite[2.3.1]{W} or \cite[Exercise 6.11]{FH}. The versions we give are stated explicitly, require no assumptions about characteristic, and  are precisely the results that we use in the proof of Lemma~\ref{key}. In fact, our proof of Proposition~\ref{87.5} was created by starting with proofs of (\ref{dentist}) and (\ref{rep-thy}).
\end{remark-no-advance}
\begin{proof} (\ref{8-21-0.a}). If $p<0$ or $q<0$, then all of the modules in \ref{8-21-0}.\ref{8-21-0.a} are zero. If $q=0$, then \ref{8-21-0}.\ref{8-21-0.a} is 
$$0\to  \textstyle \chi(p=0)\bigwedge^rF'
\xrightarrow{\incl^\dagger} \chi(p=0){\textstyle\bigwedge^r}F \xrightarrow{\quot^\dagger} 
\chi(p=0) \bigwedge^{r-1}F'\otimes  F''\to 0,$$ which is exact for all $p$. (Notice that we used the factor $\chi((p,q)\neq (0,0))$ in our proof that \ref{8-21-0}.\ref{8-21-0.a} is exact for all $p$ when $q=0$.)
 If $p=0$, then  \ref{8-21-0}.\ref{8-21-0.a} 
is 
{\small$$0\to \begin{matrix} \bigwedge^rF'\otimes \Sym_qG\\\oplus\\
\bigwedge^{r-1}F'\otimes \Sym_{q-1}G\otimes F''\otimes G''\end{matrix}
\xrightarrow{\incl^\dagger} \textstyle\bigwedge^rF\otimes \Sym_qG \xrightarrow{\quot^\dagger} 
\bigwedge^{r-1}F'\otimes \Sym_qG'\otimes F''\to 0,$$}

\noindent which is exact for all $q$. Henceforth, we assume $2\le p+q$. Observe that  Figure~\ref{the following picture} is a commutative diagram; each column  is split exact; and the middle two rows are split exact. The assertion follows from the snake lemma and the fact that the bottom $\incl^\dagger$ in Figure~\ref{the following picture} is automatically an injection. 

(\ref{8-21-0.b}). If $p<0$ or $q<0$, then all of the modules in \ref{8-21-0}.\ref{8-21-0.b} are zero. If $q=0$, then \ref{8-21-0}.\ref{8-21-0.b} is 
$$\textstyle 0\to{\begin{matrix} \bigwedge^{r}F'\otimes \bigwedge^p({G'}^*)\\\oplus\\
\bigwedge^{r}F'\otimes \bigwedge^{p-1}({G'}^*)\otimes {G''}^*\end{matrix}}
\xrightarrow{\incl^\ddag} \bigwedge^{r}F\otimes \bigwedge^p(G^*)\xrightarrow{\quot^\ddag}
{\begin{matrix} 
\bigwedge^{r-1}F'\otimes \bigwedge^{p}(G^*)\otimes F''\end{matrix}}
\to0,
$$
which is exact for all integers $p$. If $p=0$, then \ref{8-21-0}.\ref{8-21-0.b} is
$$\textstyle 0\to \chi(q=0)\bigwedge^{r}F' 
\xrightarrow{\incl^\ddag} \chi(q=0)\bigwedge^{r}F \xrightarrow{\quot^\ddag}
\chi(q=0)\bigwedge^{r-1}F'\otimes  F''
\to0,$$
which is exact for all integers $q$. 
 Henceforth, we assume $2\le p+q$. Observe that  (\ref{not the following picture}) is a commutative diagram; each column  is split exact; and the middle two rows are split exact. Once again, the assertion follows from the snake lemma and the fact that the bottom $\incl^\ddag$ in Figure~\ref{not the following picture} is automatically an injection. \end{proof}

\medskip
The complexes $\mathbb L^{N,p}_\Phi$ and $\mathbb K^{N,p}_\Phi$ may be found in Definition~\ref{8mathbb-E}.\ref{8mathbb-E.c}; the complexes $\mathbb A(\Phi'')$ and $\mathbb B(\Phi'')$ are defined in Notation~\ref{notation-9}; and some comments about the total complex (or mapping cone) of a map of complexes is in \ref{Tot}.

\begin{proposition} \label{87.5}Adopt Data~{\rm\ref{8data-8-8}} and Notation~{\rm\ref{notation-9}}. Let $p$ and $N$ be integers. The following statements hold.
\begin{enumerate}[\rm(a)] 
\item\label{87.5.b} Assume that $p\neq 0$.
Then there is a canonical short exact sequence  of complexes 
\begin{align}0\to 
\Tot\Big(\mathbb L^{N,p}_{\pi'\circ \Phi}\otimes \mathbb A(\Phi'')\Big)\xrightarrow{\incl^\dagger} \mathbb L^{N,p}_\Phi\xrightarrow{\quot^\dagger} 
{\begin{matrix}\mathbb L^{N-1,p}_{\Phi'}\otimes F''\\\oplus\\   \mathbb L^{N-1,p-1}_{\Phi'}\otimes {F''} \otimes G''\end{matrix}}\to 0.
\label{87.5.gtsL}\end{align}
In particular, if $\Phi''$ is an isomorphism, then  $\quot^\dagger$ induces an isomorphism
\begin{equation}\label{87.5.cncl'}\HH_j(\mathbb L^{N-1,p}_{\Phi'})\oplus \HH_j(\mathbb L^{N-1,p-1}_{\Phi'})\cong
\HH_j(\mathbb L^{N,p}_\Phi),\end{equation} for all $j$.

\item\label{87.5.a} Assume that   
$p\neq 0$. 
Then there is a canonical short exact sequence  of complexes 
\begin{align}0\to{\begin{matrix}\mathbb K^{N,p}_{\Phi'}\\\oplus\\  (\mathbb K^{N,p-1}_{\Phi'}\otimes {G''}^*)\end{matrix}}\xrightarrow{\incl^\ddag} \mathbb K^{N,p}_\Phi\xrightarrow{\quot^\ddag}
\Tot\Big(\mathbb K^{N-1,p}_{\pi'\circ \Phi} \otimes 
\mathbb B(\Phi'')
\Big)\to 0.
\label{87.5.gts}\end{align}
In particular, if $\Phi''$ is an isomorphism, then  $\incl^\ddag$ induces an isomorphism
\begin{equation}\label{87.5.cncl}\HH_j(\mathbb K^{N,p}_{\Phi'})\oplus \HH_j(\mathbb K^{N,p-1}_{\Phi'})\cong \HH_j(\mathbb K^{N,p}_\Phi),\end{equation} for all $j$.
\end{enumerate}
\end{proposition}

\begin{remark-no-advance} 
 The conclusion of Proposition~\ref{87.5}.\ref{87.5.a} does not hold when $p=0$ and $1\le N\le \goth f$. In this case, $\mathbb K_{\Phi'}^{N,p}$ and $\mathbb K_{\Phi}^{N,p}$ are $0\to \bigwedge^NF'\to 0$ and $0\to \bigwedge^NF\to 0$, respectively, with each non-zero module in position zero, and $\mathbb K_{\Phi'}^{N,p-1}$ is the zero complex. Thus, (\ref{87.5.cncl}) does not hold. 

The conclusion of Proposition~\ref{87.5}.\ref{87.5.b} also does not hold when $p=0$. In particular, $\mathbb L^{N,0}_\Phi$ is the complex 
$$\textstyle 0\to \bigwedge^{N+1}F\otimes \Sym_1G\xrightarrow{\Kos_\Phi}\bigwedge^{N+2}F\otimes \Sym_2G\xrightarrow{\Kos_\Phi}\dots \xrightarrow{\Kos_\Phi}\bigwedge^{\goth f}F\otimes \Sym_{\goth f-N}G\to 0.
$$ In the typical situation $$\textstyle \HH_{\goth f-N-1}(\mathbb L^{N,0}_\Phi)=\bigwedge ^NF, \quad    \HH_{\goth f-N-1}(\mathbb L^{N-1,0}_{\Phi'})=\bigwedge ^{N-1}F',$$ and (\ref{87.5.cncl'}) fails to hold.
\end{remark-no-advance}

\setcounter{figure}{\value{equation}}
\begin{figure}
\begin{center}
{\footnotesize$$ \xymatrix{
{\begin{matrix} \bigwedge^{r}F'\otimes L^p_qG\\\oplus\\
\bigwedge^{r-1}F'\otimes L^p_{q-1}G\otimes F''\otimes G''\end{matrix}}
\ar[r]^(.6){\incl^\dagger}\ar[d]^{\left[\smallmatrix \Kos_{\pi'\circ \Phi}&0\\\Kos_{\Phi''}&-\Kos_{\pi'\circ \Phi}\endsmallmatrix\right]}& \bigwedge^{r}F\otimes L^p_qG\ar[r]^(.37){\quot^\dagger}\ar[d]^{\Kos_\Phi}&
{\begin{matrix} 
\bigwedge^{r-1}F'\otimes L^p_qG'\otimes F''\\\oplus\\
\bigwedge^{r-1}F'\otimes L^{p-1}_{q}G'\otimes F''\otimes G''\end{matrix}}
\ar[d]^{\left[\smallmatrix -\Kos_{\Phi'}&0\\0&-\Kos_{\Phi'}\endsmallmatrix\right]}
\\
{\begin{matrix} \bigwedge^{r+1}F'\otimes L^p_{q+1}G\\\oplus\\
\bigwedge^{r}F'\otimes L^p_{q}G\otimes F''\otimes G''\end{matrix}}
\ar[r]^(.6){\incl^\dagger}& \bigwedge^{r+1}F\otimes L^p_{q+1}G\ar[r]^(.37){\quot^\dagger}&
{\begin{matrix} 
\bigwedge^{r}F'\otimes L^p_{q+1}G'\otimes F''\\\oplus\\
\bigwedge^{r}F'\otimes L^{p-1}_{q+1}G'\otimes F''\otimes G''\end{matrix}}
}$$}
\caption{In the proof of Proposition~\ref{87.5}.\ref{87.5.b} one 
verifies that this diagram (with $p$ and $q$ both positive) commutes in order  to see that that 
 (\ref{87.5.gtsL})  is a map of complexes.}\setcounter{figure}{\value{equation}}\label{Aug-23-2}\addtocounter{equation}{1}\setcounter{figure}{\value{equation}}
{\footnotesize$$ \xymatrix{
{\begin{matrix} \bigwedge^{r}F'\otimes K^p_q({G'}^*)\\\oplus\\
\bigwedge^{r}F'\otimes K^{p-1}_{q}({G'}^*)\otimes {G''}^*\end{matrix}}
\ar[r]^(.63){\incl^\ddag}\ar[d]^{\left[\smallmatrix \eta_{\pi'\circ \Phi}&0\\0&\eta_{\pi'\circ \Phi}\endsmallmatrix\right]}& \bigwedge^{r}F\otimes K^p_q(G^*)\ar[r]^(.4){\quot^\ddag}\ar[d]^{\eta_\Phi}&
{\begin{matrix} 
\bigwedge^{r}F'\otimes K^p_{q-1}(G^*)\otimes {G''}^*\\\oplus\\
\bigwedge^{r-1}F'\otimes K^{p}_{q}(G^*)\otimes F''\end{matrix}}
\ar[d]^{\left[\smallmatrix \eta_{\Phi'}&0\\\Phi''&-\eta_{\Phi'}\endsmallmatrix\right]}
\\
{\begin{matrix} \bigwedge^{r+1}F'\otimes K^p_{q-1}(G^*)\\\oplus\\
\bigwedge^{r+1}F'\otimes K^{p-1}_{q-1}(G^*)\otimes F''\otimes {G''}^*\end{matrix}}
\ar[r]^(.63){\incl^\ddag}& \bigwedge^{r+1}F\otimes K^p_{q-1}(G^*)\ar[r]^(.4){\quot^\ddag}&
{\begin{matrix} 
\bigwedge^{r+1}F'\otimes K^p_{q-2}(G^*)\otimes {G''}^*\\\oplus\\
\bigwedge^{r}F'\otimes K^{p}_{q-1}(G^*)\otimes F''\end{matrix}}
}$$}
\caption{In the proof of Proposition~\ref{87.5}.\ref{87.5.a} one 
verifies that this diagram commutes in order  to see that that 
 (\ref{87.5.gts})  is a map of complexes.} \setcounter{figure}{\value{equation}}\label{Aug-23-3}\addtocounter{equation}{1}
\end{center}\end{figure}

\begin{proof} (\ref{87.5.b}). There is nothing to prove if $p\le 0$; so we assume that  $1\le p$.   
Notice that (\ref{87.5.gtsL}) in position $j$ is
{\tiny $$0\to \begin{matrix} \bigwedge^{\goth f-j}F'\otimes L^p_{\goth f-j-N}G\\\oplus\\
\bigwedge^{\goth f-j-1}F'\otimes L^p_{\goth f-j-N-1}G\otimes F''\otimes G''\end{matrix}
\xrightarrow{\incl^\dagger} {\textstyle\bigwedge^{\goth f-j}}F\otimes L^p_{\goth f-j-N}G \xrightarrow{\quot^\dagger} 
\begin{matrix} \bigwedge^{\goth f-j-1}F'\otimes L^p_{\goth f-j-N}G'\otimes F''\\\oplus\\
\chi((p,j)\neq (1,\goth f-N))\bigwedge^{\goth f-j-1}F'\otimes L^{p-1}_{\goth f-j-N}G'\otimes F''\otimes G''\end{matrix}\to 0,$$}

\noindent which is \ref{8-21-0}.\ref{8-21-0.a} with $r$ replaced by $\goth f-j$ and $q$ replaced by $\goth f-j-N$; 
hence, each row of (\ref{87.5.gtsL}) is a short exact sequence. To see that that 
 (\ref{87.5.gtsL})  is a map of complexes, one  verifies that
  Figure~\ref{Aug-23-2} is a commutative diagram, and this is  straightforward. Thus, 
(\ref{87.5.gtsL}) is a short exact sequence of complexes.

The assertion (\ref{87.5.cncl'}) is now obvious. Indeed, the hypotheses that $\Phi''$ is an isomorphism ensures that the complex $\Tot\Big(\mathbb L^{N,p}_{\pi'\circ \Phi}\otimes \mathbb A(\Phi'')\Big)$
on the left side of (\ref{87.5.gtsL}) has homology zero and the long exact sequence of homology associated to the short exact sequence of complexes (\ref{87.5.gtsL}) yields (\ref{87.5.cncl'}). 

\medskip\noindent(\ref{87.5.a}). Assume that $1\le p$. Notice that (\ref{87.5.gts}) in position $j$ is
{\footnotesize $$0\to \begin{matrix} \bigwedge^{N-j}F'\otimes K^p_{j}({G'}^*)\\\oplus\\
\bigwedge^{N-j}F'\otimes K^{p-1}_{j}({G'}^*)\otimes {G''}^*\end{matrix}
\xrightarrow{\incl^\ddag} {\textstyle\bigwedge^{N-j}}F\otimes K^p_{j}(G^*) \xrightarrow{\quot^\ddag} 
\begin{matrix} 
\bigwedge^{N-j}F'\otimes K^{p}_{j-1}(G^*)\otimes {G''}^*
\\\oplus\\
\bigwedge^{N-j-1}F'\otimes K^p_{j}(G^*)\otimes F''
\end{matrix}\to 0,$$}

\noindent which is \ref{8-21-0}.\ref{8-21-0.b} with $r$ replaced by $N-j$ and $q$ replaced by $j$. The parameter $p$ is is not zero by hypothesis; hence, 
Observation~\ref{8-21-0}.\ref{8-21-0.b} ensures that
each row of (\ref{87.5.gts}) is a short exact sequence. To see that that 
 (\ref{87.5.gts})  is a map of complexes, one  verifies that
  Figure~\ref{Aug-23-3} is a commutative diagram, and this is  straightforward. Thus, 
(\ref{87.5.gts}) is a short exact sequence of complexes.
The assertion (\ref{87.5.cncl}) is now obvious, as in the proof of (\ref{87.5.cncl'}). 
\end{proof}

\bigskip

\section{The definition and elementary properties of the complexes $\mathcal C_\Phi^{i,a}$.}

\bigskip
The maps and modules of $\mathcal C_\Phi^{i,a}$ are introduced in \ref{CiaPhi}. It is shown in \ref{complex} that each $\mathcal C_\Phi^{i,a}$ is a complex. Examples are given in \ref{3.3} and \ref{5.6}. The relationship between the classical generalized Eagon-Northcott complexes $\{\mathcal C_\Phi^{i}\}$ and the complexes $\{\mathcal C_\Phi^{i,a}\}$ is examined in \ref{new=old}. The duality in the family $\{\mathcal C_\Phi^{i,a}\}$ is studied in \ref{duality}. Information about the length of $\mathcal C_\Phi^{i,a}$ is recorded in \ref{3.7}.
The zero-th homology of $\mathcal C_\Phi^{i,a}$, for $-1\le i$, may be found in \ref{H0}. The fact that $I_\goth g(\Phi)$ annihilates $\HH_0(\mathcal C_\Phi^{i,a})$, for $-1\le i$, is established in \ref{ann(H0)}.

\begin{data}\label{setup}Let $R$ be a commutative Noetherian ring, $F$ and $G$ be free $R$-modules of rank $\goth f$ and $\goth g$, respectively, with $\goth g\le \goth f$, $\Phi:G^*\to F$ be an $R$-module homomorphism.    
\end{data}
\begin{definition}\label{CiaPhi}Adopt Data~{\rm{\ref{setup}}}. Recall the complexes $\mathbb K_{\Phi}^{N,p}$ and $\mathbb L_{\Phi}^{N,p}$ of Definition~{\rm{\ref{8mathbb-E}.\ref{8mathbb-E.c}}} and Remark~{\rm\ref{5.3}.\ref{5.3.c}} and the homomorphism $\kappa$ of \ref{2.3}. Let $i$ and $a$ be integers with $1\le a\le \goth g$. Define the maps and modules $(\mathcal C^{i,a}_\Phi,d)$ to be 
\begin{equation}\label{FirstRecipe}\textstyle 0\to \mathbb K^{\goth f-\goth g-i-1,a}_\Phi[-i-2]\xrightarrow{d} \bigwedge^{\goth f-\goth g+a-i-1}F \xrightarrow{d} \mathbb L^{\goth f-i-1,\goth g-a}_\Phi\to 0,\end{equation}with $[\mathcal C^{i,a}_\Phi]_{i+1}=\bigwedge^{f-g+a-i-1}F$. 
The differentials 
$$[\mathcal C^{i,a}_\Phi]_{i+2}\xrightarrow{d_{i+2}} [\mathcal C^{i,a}_\Phi]_{i+1}\xrightarrow{d_{i+i}}[\mathcal C^{i,a}_\Phi]_i$$
are
\begin{align*}\textstyle[\mathcal C^{i,a}_\Phi]_{i+2}&\textstyle=[\mathbb K^{\goth f-\goth g-i-1,a}_\Phi]_0=\bigwedge^{\goth f-\goth g-i-1}F\otimes K_0^a(G^*)=\bigwedge^{\goth f-\goth g-i-1}F\otimes \bigwedge^a(G^*)\\&\textstyle\xrightarrow{1\otimes \bigwedge^a \Phi}\bigwedge^{\goth f-\goth g-i-1}F\otimes \bigwedge^aF\xrightarrow{\mult}\bigwedge^{\goth f-\goth g+a-i-1}F=[\mathcal C^{i,a}_\Phi]_{i+1}\\\intertext{and}
[\mathcal C^{i,a}_\Phi]_{i+1}&\textstyle=\bigwedge^{\goth f-\goth g+a-i-1}F
\xrightarrow{1\otimes \ev^*(1)}\bigwedge^{\goth f-\goth g+a-i-1}F\otimes \bigwedge^{\goth g+1-a}(G^*)\otimes \bigwedge^{\goth g+1-a}G\\&\textstyle\xrightarrow{1\otimes \bigwedge^{\goth g+1-a}\Phi\otimes \kappa}
\bigwedge^{\goth f-\goth g+a-i-1}F\otimes \bigwedge^{\goth g+1-a}F\otimes L_1^{\goth g-a}G\\&\textstyle\xrightarrow{\mult\otimes 1}
\bigwedge^{\goth f-i}F\otimes L_1^{\goth g-a}G
=[\mathbb L^{\goth f-i-1,\goth g-a}_\Phi]_i=[\mathcal C^{i,a}_\Phi]_i.
\end{align*}
\end{definition}

\begin{remark}\label{SecondRecipe}We give two other descriptions of the modules in $\mathcal C^{i,a}_\Phi$:
\begin{equation}\label{spot-by-spot}[\mathcal C^{i,a}_\Phi]_j=\begin{cases}
\bigwedge^{\goth f-j}F\otimes L_{i+1-j}^{\goth g -a}G,&\text{if $j\le i$,}\\
\bigwedge^{\goth f-\goth g+a-i-1}F,&\text{if $j=i+1$, and}\\
\bigwedge^{\goth f-\goth g+1-j}F\otimes K_{j-i-2}^{a}(G^*),&\text{if $i+2\le j$,}\end{cases}\end{equation}
and $\mathcal C^{i,a}_\Phi$ looks like 
\begin{align}\label{all-at-once}\textstyle \cdots \xrightarrow{\eta_\Phi} \bigwedge ^{\goth f-\goth g-i-2}F\otimes K_1^a (G^*) \xrightarrow{\eta_\Phi} \bigwedge ^{\goth f-\goth g-i-1}F\otimes K_0^a (G^*) \xrightarrow{\bigwedge^a\Phi} \bigwedge ^{\goth f-\goth g+a-i-1}F\vspace{5pt} \\\textstyle
\xrightarrow{\bigwedge^{\goth g+1-a}\Phi} \bigwedge ^{\goth f-i}F
\otimes L_1^{\goth g-a} G \xrightarrow{\Kos_\Phi} 
\bigwedge ^{\goth f-i+1}F
\otimes L_2^{\goth g-a} G \xrightarrow{\Kos_\Phi} \cdots,\notag\end{align}
with  $\bigwedge ^{\goth f-\goth g+a-i-1}F$ in position $i+1$. 
\end{remark}

\begin{observation}\label{complex} Adopt Data~{\rm\ref{setup}}. Let and $i$ and $a$ be integers with $1\le a\le \goth g$.
Then the maps and modules $(\mathcal C_{\Phi}^{i,a},d)$ of Definition~{\rm{\ref{CiaPhi}}} form a complex. 
\end{observation}
\begin{proof}One obtains $\mathcal C_{\Phi}^{i,a}$ by pasting together two well-known complexes; hence it suffices to show that 
$$(\mathcal C_{\Phi}^{i,a})_{i+3}\xrightarrow{d_{i+3}}(\mathcal C_{\Phi}^{i,a})_{i+2}\xrightarrow{d_{i+2}}
(\mathcal C_{\Phi}^{i,a})_{i+1}\xrightarrow{d_{i+1}}
(\mathcal C_{\Phi}^{i,a})_{i}\xrightarrow{d_{i}}
(\mathcal C_{\Phi}^{i,a})_{i-1}$$is a complex; furthermore, by the duality of Observation~\ref{duality}, it suffices to show that $$
(\mathcal C_{\Phi}^{i,a})_{i+2}\xrightarrow{d_{i+2}}
(\mathcal C_{\Phi}^{i,a})_{i+1}\xrightarrow{d_{i+1}}
(\mathcal C_{\Phi}^{i,a})_{i}\xrightarrow{d_{i}}
(\mathcal C_{\Phi}^{i,a})_{i-1}$$is a complex. Take $f\in \bigwedge^{\goth f-\goth g+a-i-1}F=(\mathcal C_{\Phi}^{i,a})_{i+1}$. We compute
\begingroup\allowdisplaybreaks\begin{align*}(d_i\circ d_{i+1})(f)&= \textstyle d_i\left(\sum_\ell
f\wedge(\bigwedge^{\goth g +1-a}\Phi)(m^*_\ell)\otimes \kappa(m_\ell)
\in \bigwedge^{\goth f-i}F\otimes L^{\goth g-a}_1 
G=(\mathcal C_{\Phi}^{i,a})_{i}
\right)\\
&= \textstyle \sum_\ell\sum_{\ell'}
f\wedge(\bigwedge^{\goth g +1-a}\Phi)(m^*_\ell)\wedge \Phi(n^*_{\ell'})\otimes \kappa(m_\ell)\cdot(1\otimes n_{\ell'})\\
&\hskip20pt\textstyle  \in \bigwedge^{\goth f-i+1}F\otimes L^{\goth g-a}_2 
G=(\mathcal C_{\Phi}^{i,a})_{i-1}
\\
&= \textstyle \sum_\ell\sum_{\ell'}\sum_{\ell''}
f\wedge(\bigwedge^{\goth g +2-a}\Phi)(m^*_\ell\wedge n^*_{\ell'})\otimes  n_{\ell''}^*(m_\ell)\otimes n_{\ell''}\cdot n_{\ell'}\\
&= \textstyle \sum_L\sum_{\ell'}\sum_{\ell''}
f\wedge(\bigwedge^{\goth g +2-a}\Phi)(M^*_L)\otimes  n_{\ell''}^*(n_{\ell'}^*(M))\otimes n_{\ell''}\cdot n_{\ell'}\\&=0,
\end{align*}\endgroup
where \begingroup\allowdisplaybreaks\begin{align*}\ev^*(1)&\textstyle =\sum_{\ell} m_{\ell}^*\otimes m_{\ell} &&\textstyle \in \bigwedge^{\goth g+1-a}G^*\otimes \bigwedge^{\goth g+1-a}G,\\
\ev^*(1)&\textstyle =\sum_{\ell'} n_{\ell'}^*\otimes n_{\ell'}
=\sum_{\ell''} n_{\ell''}^*\otimes n_{\ell''}
 &&\textstyle \in G^*\otimes G, \text{ and}\\
\ev^*(1)&\textstyle =\sum_{L} M_{L}^*\otimes M_{L} &&\textstyle \in \bigwedge^{\goth g+2-a}G^*\otimes \bigwedge^{\goth g+2-a}G
\end{align*}\endgroup
are the canonical elements of (\ref{ev-star}). One easily verifies that
\begin{equation}\notag\textstyle \sum_\ell\sum_{\ell'}m_\ell^*\wedge n_{\ell'}^*\otimes m_{\ell}
\otimes n_{\ell'}=\sum_L\sum_{\ell'}M_L^*\otimes n_{\ell'}^*(M_{L})
\otimes n_{\ell'}\in \bigwedge^{\goth g+2-a}G^*\otimes \bigwedge^{\goth g+1-a}G \otimes G.\end{equation}(Merely evaluate both sides at a typical element of $\bigwedge^{\goth g+2-a}G\otimes \bigwedge^{\goth g+1-a}G^* \otimes G^*$.) It is obvious that 
\begin{equation}\notag
\textstyle\sum_{\ell'}\sum_{\ell''} n_{\ell'}^*\wedge n_{\ell''}^*\otimes n_{\ell'}\cdot n_{\ell''}=0\in \bigwedge^2G^*\otimes \Sym_2 G.\end{equation}

Take $f\in \bigwedge^{\goth f-\goth g-i-1}F$ and $\gamma\in \bigwedge^a(G^*)=K_0^a(G^*)$, see \ref{2.8}.\ref{2.8.h}. So, $f\otimes \gamma$ is in $(\mathcal C_\Phi^{i,a})_{i+2}$. We compute
\begingroup\allowdisplaybreaks\begin{align*}(d_{i+1}\circ d_{i+2})(f\otimes \gamma)&\textstyle= d_{i+1}\left(f\wedge (\bigwedge^a\Phi)(\gamma)\in \bigwedge^{\goth f-\goth g+a-i-1}F=(\mathcal C_\Phi^{i,a})_{i+1}\right)\\
&\textstyle= \sum_{\ell}f\wedge (\bigwedge^a\Phi)(\gamma)\wedge (\bigwedge^{\goth g-a+1}\Phi)(m_\ell^*)\otimes \kappa(m_{\ell})\\&\textstyle\hskip20pt\in \bigwedge^{\goth f-i}F\otimes L^{\goth g-a}_1G=(\mathcal C_\Phi^{i,a})_{i}\\
&\textstyle= \sum_{\ell}f\wedge (\bigwedge^{\goth g+1}\Phi)(\gamma\wedge m_\ell^*)\otimes \kappa(m_{\ell})\\&=0.
\end{align*}\endgroup 
\vskip-20pt\end{proof}

\begin{remark}\label{July-6} If $a$ is equal to $0$ or $\goth g+1$, then it is possible to construct a complex $\mathcal C^{i,a}_{\Phi}$ using the recipe of Remark~\ref{SecondRecipe}. These complexes are 
$$\textstyle \mathcal C^{i,0}_{\Phi}:\quad 0\to \bigwedge^{\goth f-\goth g-i-1}F \xrightarrow{\text{identity map}} \bigwedge^{\goth f-\goth g-i-1}F\to 0,$$ with the non-zero modules appearing in positions $i+2$ and $i+1$; 
and
$$\textstyle \mathcal C^{i,\goth g+1}_{\Phi}:\quad 0\to  \bigwedge^{\goth f-i}F\to 0,$$ with the non-zero module appearing in position $i+1$.
\end{remark}

\begin{example}\label{3.3} Adopt Data~{\rm\ref{setup}}  
with  $(\goth g,\goth f)=(5,9)$. We record the complexes $\mathcal C_{\Phi}^{i,a}$ of Definition~{\rm\ref{CiaPhi}} and Observation~{\rm\ref{complex}} which have the form:
$$0\to (\mathcal C_{\Phi}^{i,a})_5\to (\mathcal C_{\Phi}^{i,a})_4\to (\mathcal C_{\Phi}^{i,a})_3\to (\mathcal C_{\Phi}^{i,a})_2\to (\mathcal C_{\Phi}^{i,a})_1
\to (\mathcal C_{\Phi}^{i,a})_0\to 0.$$Of course, these complexes have length 
$\goth f-\goth g+1=5$.

\begingroup\allowdisplaybreaks\begin{align}\scriptscriptstyle &\scriptscriptstyle \mathcal C_{\Phi}^{-1,1}:\ 0\to\bigwedge^0F\otimes K^1_{4}(G^*)\xrightarrow{\eta_\Phi} \bigwedge^1F\otimes K^1_{3}(G^*)\xrightarrow{\eta_\Phi}\bigwedge^2F\otimes K^1_{2}(G^*)\xrightarrow{\eta_\Phi} \bigwedge^3F\otimes K^1_{1}(G^*)\xrightarrow{\eta_\Phi}
\bigwedge^4F\otimes K^1_{0}(G^*)\xrightarrow{\bigwedge^1\Phi} \bigwedge ^5F
\to 0 \notag\\
&\scriptscriptstyle \mathcal C_{\Phi}^{0,1}:\ 0\to\bigwedge^0F\otimes K^1_{3}(G^*)\xrightarrow{\eta_\Phi} \bigwedge^1F\otimes K^1_{2}(G^*)\xrightarrow{\eta_\Phi}\bigwedge^2F\otimes K^1_{1}(G^*)\xrightarrow{\eta_\Phi} \bigwedge^3F\otimes K^1_{0}(G^*)\xrightarrow{\bigwedge^1\Phi}
\bigwedge^4F\xrightarrow{\bigwedge^5\Phi} \bigwedge ^9F\otimes L^4_{1}G
\to 0 \notag\\
&\scriptscriptstyle \mathcal C_{\Phi}^{1,1}:\ 0\to\bigwedge^0F\otimes K^1_{2}(G^*)\xrightarrow{\eta_\Phi} \bigwedge^1F\otimes K^1_{1}(G^*)\xrightarrow{\eta_\Phi}\bigwedge^2F\otimes K^1_{0}(G^*)\xrightarrow{\bigwedge^1\Phi} \bigwedge^3F\xrightarrow{\bigwedge^5\Phi}
\bigwedge^8F\otimes L^4_{1}G\xrightarrow{\Kos_\Phi} \bigwedge ^9F\otimes L^4_{2}G
\to 0 \notag\\
&\scriptscriptstyle \mathcal C_{\Phi}^{2,1}:\ 0\to\bigwedge^0F\otimes K^1_{1}(G^*)\xrightarrow{\eta_\Phi} \bigwedge^1F\otimes K^1_{0}(G^*)\xrightarrow{\bigwedge^1\Phi}\bigwedge^2F\xrightarrow{\bigwedge^5\Phi} \bigwedge^7F\otimes L^4_{1}G\xrightarrow{\Kos_\Phi} 
\bigwedge^8F\otimes L^4_{2}G\xrightarrow{\Kos_\Phi} \bigwedge ^9F\otimes L^4_{3}G
\to 0 \notag\\
&\scriptscriptstyle \mathcal C_{\Phi}^{3,1}:\ 0\to\bigwedge^0F\otimes K^1_{0}(G^*)\xrightarrow{\bigwedge^1\Phi} \bigwedge^1F\xrightarrow{\bigwedge^5\Phi}\bigwedge^6F\otimes L^4_{1}G\xrightarrow{\Kos_\Phi} \bigwedge^7F\otimes L^4_{2}G\xrightarrow{\Kos_\Phi}
\bigwedge^8F\otimes L^4_{3}G\xrightarrow{\Kos_\Phi} \bigwedge ^9F\otimes L^4_{4}G
\to 0 \notag\\
&\scriptscriptstyle \mathcal C_{\Phi}^{4,1}:\ 0\to\bigwedge^0F\xrightarrow{\bigwedge^5\Phi} \bigwedge^5F\otimes L^4_{1}G\xrightarrow{\Kos_\Phi}\bigwedge^6F\otimes L^4_{2}G\xrightarrow{\Kos_\Phi} \bigwedge^7F\otimes L^4_{3}G\xrightarrow{\Kos_\Phi}
\bigwedge^8F\otimes L^4_{4}G\xrightarrow{\Kos_\Phi} \bigwedge ^9F\otimes L^4_{5}G\to 0 \notag\\
&\scriptscriptstyle \mathcal C_{\Phi}^{5,1}:\ 0\to
\bigwedge^4F\otimes L^4_{1}G
\xrightarrow{\Kos_\Phi} \bigwedge^5F\otimes L^4_{2}G\xrightarrow{\Kos_\Phi}\bigwedge^6F\otimes L^4_{3}G\xrightarrow{\Kos_\Phi} \bigwedge^7F\otimes L^4_{4}G\xrightarrow{\Kos_\Phi}
\bigwedge^8F\otimes L^4_{5}G\xrightarrow{\Kos_\Phi} \bigwedge ^9F\otimes L^4_{6}G
\to 0 \notag
\\\hline
\scriptscriptstyle &\scriptscriptstyle \mathcal C_{\Phi}^{-1,2}:\ 0\to\bigwedge^0F\otimes K^2_{4}(G^*)\xrightarrow{\eta_\Phi} \bigwedge^1F\otimes K^2_{3}(G^*)\xrightarrow{\eta_\Phi}\bigwedge^2F\otimes K^2_{2}(G^*)\xrightarrow{\eta_\Phi} \bigwedge^3F\otimes K^2_{1}(G^*)\xrightarrow{\eta_\Phi}
\bigwedge^4F\otimes K^2_{0}(G^*)\xrightarrow{\bigwedge^2\Phi} \bigwedge ^6F
\to 0 \notag\\
&\scriptscriptstyle \mathcal C_{\Phi}^{0,2}:\ 0\to\bigwedge^0F\otimes K^2_{3}(G^*)\xrightarrow{\eta_\Phi} \bigwedge^1F\otimes K^2_{2}(G^*)\xrightarrow{\eta_\Phi}\bigwedge^2F\otimes K^2_{1}(G^*)\xrightarrow{\eta_\Phi} \bigwedge^3F\otimes K^2_{0}(G^*)\xrightarrow{\bigwedge^2\Phi}
\bigwedge^5F\xrightarrow{\bigwedge^4\Phi} \bigwedge ^9F\otimes L^3_{1}G
\to 0 \notag\\
&\scriptscriptstyle \mathcal C_{\Phi}^{1,2}:\ 0\to\bigwedge^0F\otimes K^2_{2}(G^*)\xrightarrow{\eta_\Phi} \bigwedge^1F\otimes K^2_{1}(G^*)\xrightarrow{\eta_\Phi}\bigwedge^2F\otimes K^2_{0}(G^*)\xrightarrow{\bigwedge^2\Phi} \bigwedge^4F\xrightarrow{\bigwedge^4\Phi}
\bigwedge^8F\otimes L^3_{1}G\xrightarrow{\Kos_\Phi} \bigwedge ^9F\otimes L^3_{2}G
\to 0 \notag\\
&\scriptscriptstyle \mathcal C_{\Phi}^{2,2}:\ 0\to\bigwedge^0F\otimes K^2_{1}(G^*)\xrightarrow{\eta_\Phi} \bigwedge^1F\otimes K^2_{0}(G^*)\xrightarrow{\bigwedge^2\Phi}\bigwedge^3F\xrightarrow{\bigwedge^4\Phi} \bigwedge^7F\otimes L^3_{1}G\xrightarrow{\Kos_\Phi} 
\bigwedge^8F\otimes L^3_{2}G\xrightarrow{\Kos_\Phi} \bigwedge ^9F\otimes L^3_{3}G
\to 0 \notag\\
&\scriptscriptstyle \mathcal C_{\Phi}^{3,2}:\ 0\to\bigwedge^0F\otimes K^2_{0}(G^*)\xrightarrow{\bigwedge^2\Phi} \bigwedge^2F\xrightarrow{\bigwedge^4\Phi}\bigwedge^6F\otimes L^3_{1}G\xrightarrow{\Kos_\Phi} \bigwedge^7F\otimes L^3_{2}G\xrightarrow{\Kos_\Phi}
\bigwedge^8F\otimes L^3_{3}G\xrightarrow{\Kos_\Phi} \bigwedge ^9F\otimes L^3_{4}G
\to 0 \notag\\
&\scriptscriptstyle \mathcal C_{\Phi}^{4,2}:\ 0\to\bigwedge^1F\xrightarrow{\bigwedge^4\Phi} \bigwedge^5F\otimes L^3_{1}G\xrightarrow{\Kos_\Phi}\bigwedge^6F\otimes L^3_{2}G\xrightarrow{\Kos_\Phi} \bigwedge^7F\otimes L^3_{3}G\xrightarrow{\Kos_\Phi}
\bigwedge^8F\otimes L^3_{4}G\xrightarrow{\Kos_\Phi} \bigwedge ^9F\otimes L^3_{5}G\to 0 \notag\\\hline
\scriptscriptstyle &\scriptscriptstyle \mathcal C_{\Phi}^{-1,3}:\ 0\to\bigwedge^0F\otimes K^3_{4}(G^*)\xrightarrow{\eta_\Phi} \bigwedge^1F\otimes K^3_{3}(G^*)\xrightarrow{\eta_\Phi}\bigwedge^2F\otimes K^3_{2}(G^*)\xrightarrow{\eta_\Phi} \bigwedge^3F\otimes K^3_{1}(G^*)\xrightarrow{\eta_\Phi}
\bigwedge^4F\otimes K^3_{0}(G^*)\xrightarrow{\bigwedge^3\Phi} \bigwedge ^7F
\to 0 \notag\\
&\scriptscriptstyle \mathcal C_{\Phi}^{0,3}:\ 0\to\bigwedge^0F\otimes K^3_{3}(G^*)\xrightarrow{\eta_\Phi} \bigwedge^1F\otimes K^3_{2}(G^*)\xrightarrow{\eta_\Phi}\bigwedge^2F\otimes K^3_{1}(G^*)\xrightarrow{\eta_\Phi} \bigwedge^3F\otimes K^3_{0}(G^*)\xrightarrow{\bigwedge^3\Phi}
\bigwedge^6F\xrightarrow{\bigwedge^3\Phi} \bigwedge ^9F\otimes L^2_{1}G
\to 0 \notag\\
&\scriptscriptstyle \mathcal C_{\Phi}^{1,3}:\ 0\to\bigwedge^0F\otimes K^3_{2}(G^*)\xrightarrow{\eta_\Phi} \bigwedge^1F\otimes K^3_{1}(G^*)\xrightarrow{\eta_\Phi}\bigwedge^2F\otimes K^3_{0}(G^*)\xrightarrow{\bigwedge^3\Phi} \bigwedge^5F\xrightarrow{\bigwedge^3\Phi}
\bigwedge^8F\otimes L^2_{1}G\xrightarrow{\Kos_\Phi} \bigwedge ^9F\otimes L^2_{2}G
\to 0 \notag\\
&\scriptscriptstyle \mathcal C_{\Phi}^{2,3}:\ 0\to\bigwedge^0F\otimes K^3_{1}(G^*)\xrightarrow{\eta_\Phi} \bigwedge^1F\otimes K^3_{0}(G^*)\xrightarrow{\bigwedge^3\Phi}\bigwedge^4F\xrightarrow{\bigwedge^3\Phi} \bigwedge^7F\otimes L^2_{1}G\xrightarrow{\Kos_\Phi} 
\bigwedge^8F\otimes L^2_{2}G\xrightarrow{\Kos_\Phi} \bigwedge ^9F\otimes L^2_{3}G
\to 0 \notag\\
&\scriptscriptstyle \mathcal C_{\Phi}^{3,3}:\ 0\to\bigwedge^0F\otimes K^3_{0}(G^*)\xrightarrow{\bigwedge^3\Phi} \bigwedge^3F\xrightarrow{\bigwedge^3\Phi}\bigwedge^6F\otimes L^2_{1}G\xrightarrow{\Kos_\Phi} \bigwedge^7F\otimes L^2_{2}G\xrightarrow{\Kos_\Phi}
\bigwedge^8F\otimes L^2_{3}G\xrightarrow{\Kos_\Phi} \bigwedge ^9F\otimes L^2_{4}G
\to 0 \notag\\
&\scriptscriptstyle \mathcal C_{\Phi}^{4,3}:\ 0\to\bigwedge^2F\xrightarrow{\bigwedge^3\Phi} \bigwedge^5F\otimes L^2_{1}G\xrightarrow{\Kos_\Phi}\bigwedge^6F\otimes L^2_{2}G\xrightarrow{\Kos_\Phi} \bigwedge^7F\otimes L^2_{3}G\xrightarrow{\Kos_\Phi}
\bigwedge^8F\otimes L^2_{4}G\xrightarrow{\Kos_\Phi} \bigwedge ^9F\otimes L^2_{5}G\to 0 \notag\\ \hline
\scriptscriptstyle &\scriptscriptstyle \mathcal C_{\Phi}^{-1,4}:\ 0\to\bigwedge^0F\otimes K^4_{4}(G^*)\xrightarrow{\eta_\Phi} \bigwedge^1F\otimes K^4_{3}(G^*)\xrightarrow{\eta_\Phi}\bigwedge^2F\otimes K^4_{2}(G^*)\xrightarrow{\eta_\Phi} \bigwedge^3F\otimes K^4_{1}(G^*)\xrightarrow{\eta_\Phi}
\bigwedge^4F\otimes K^4_{0}(G^*)\xrightarrow{\bigwedge^4\Phi} \bigwedge ^8F
\to 0 \notag\\
&\scriptscriptstyle \mathcal C_{\Phi}^{0,4}:\ 0\to\bigwedge^0F\otimes K^4_{3}(G^*)\xrightarrow{\eta_\Phi} \bigwedge^1F\otimes K^4_{2}(G^*)\xrightarrow{\eta_\Phi}\bigwedge^2F\otimes K^4_{1}(G^*)\xrightarrow{\eta_\Phi} \bigwedge^3F\otimes K^4_{0}(G^*)\xrightarrow{\bigwedge^4\Phi}
\bigwedge^7F\xrightarrow{\bigwedge^2\Phi} \bigwedge ^9F\otimes L^1_{1}G
\to 0 \notag\\
&\scriptscriptstyle \mathcal C_{\Phi}^{1,4}:\ 0\to\bigwedge^0F\otimes K^4_{2}(G^*)\xrightarrow{\eta_\Phi} \bigwedge^1F\otimes K^4_{1}(G^*)\xrightarrow{\eta_\Phi}\bigwedge^2F\otimes K^4_{0}(G^*)\xrightarrow{\bigwedge^4\Phi} \bigwedge^6F\xrightarrow{\bigwedge^2\Phi}
\bigwedge^8F\otimes L^1_{1}G\xrightarrow{\Kos_\Phi} \bigwedge ^9F\otimes L^1_{2}G
\to 0 \notag\\
&\scriptscriptstyle \mathcal C_{\Phi}^{2,4}:\ 0\to\bigwedge^0F\otimes K^4_{1}(G^*)\xrightarrow{\eta_\Phi} \bigwedge^1F\otimes K^4_{0}(G^*)\xrightarrow{\bigwedge^4\Phi}\bigwedge^5F\xrightarrow{\bigwedge^2\Phi} \bigwedge^7F\otimes L^1_{1}G\xrightarrow{\Kos_\Phi} 
\bigwedge^8F\otimes L^1_{2}G\xrightarrow{\Kos_\Phi} \bigwedge ^9F\otimes L^1_{3}G
\to 0 \notag\\
&\scriptscriptstyle \mathcal C_{\Phi}^{3,4}:\ 0\to\bigwedge^0F\otimes K^4_{0}(G^*)\xrightarrow{\bigwedge^4\Phi} \bigwedge^4F\xrightarrow{\bigwedge^2\Phi}\bigwedge^6F\otimes L^1_{1}G\xrightarrow{\Kos_\Phi} \bigwedge^7F\otimes L^1_{2}G\xrightarrow{\Kos_\Phi}
\bigwedge^8F\otimes L^1_{3}G\xrightarrow{\Kos_\Phi} \bigwedge ^9F\otimes L^1_{4}G
\to 0 \notag\\
&\scriptscriptstyle \mathcal C_{\Phi}^{4,4}:\ 0\to\bigwedge^3F\xrightarrow{\bigwedge^2\Phi} \bigwedge^5F\otimes L^1_{1}G\xrightarrow{\Kos_\Phi}\bigwedge^6F\otimes L^1_{2}G\xrightarrow{\Kos_\Phi} \bigwedge^7F\otimes L^1_{3}G\xrightarrow{\Kos_\Phi}
\bigwedge^8F\otimes L^1_{4}G\xrightarrow{\Kos_\Phi} \bigwedge ^9F\otimes L^1_{5}G\to 0 \notag\\ \hline
\scriptscriptstyle 
&\scriptscriptstyle \mathcal C_{\Phi}^{-2,5}:\ 0\to\bigwedge^1F\otimes K^5_{5}(G^*)\xrightarrow{\eta_\Phi} \bigwedge^2F\otimes K^5_{4}(G^*)\xrightarrow{\eta_\Phi}\bigwedge^3F\otimes K^5_{3}(G^*)\xrightarrow{\eta_\Phi} \bigwedge^4F\otimes K^5_{2}(G^*)\xrightarrow{\eta_\Phi}
\bigwedge^5F\otimes K^5_{1}(G^*)\xrightarrow{\eta_\Phi} \bigwedge^6F\otimes K^5_{0}(G^*)
\to 0 \notag\\
&\scriptscriptstyle \mathcal C_{\Phi}^{-1,5}:\ 0\to\bigwedge^0F\otimes K^5_{4}(G^*)\xrightarrow{\eta_\Phi} \bigwedge^1F\otimes K^5_{3}(G^*)\xrightarrow{\eta_\Phi}\bigwedge^2F\otimes K^5_{2}(G^*)\xrightarrow{\eta_\Phi} \bigwedge^3F\otimes K^5_{1}(G^*)\xrightarrow{\eta_\Phi}
\bigwedge^4F\otimes K^5_{0}(G^*)\xrightarrow{\bigwedge^5\Phi} \bigwedge ^9F
\to 0 \notag\\
&\scriptscriptstyle \mathcal C_{\Phi}^{0,5}:\ 0\to\bigwedge^0F\otimes K^5_{3}(G^*)\xrightarrow{\eta_\Phi} \bigwedge^1F\otimes K^5_{2}(G^*)\xrightarrow{\eta_\Phi}\bigwedge^2F\otimes K^5_{1}(G^*)\xrightarrow{\eta_\Phi} \bigwedge^3F\otimes K^5_{0}(G^*)\xrightarrow{\bigwedge^5\Phi}
\bigwedge^8F\xrightarrow{\bigwedge^1\Phi} \bigwedge ^9F\otimes L^0_{1}G
\to 0 \notag\\
&\scriptscriptstyle \mathcal C_{\Phi}^{1,5}:\ 0\to\bigwedge^0F\otimes K^5_{2}(G^*)\xrightarrow{\eta_\Phi} \bigwedge^1F\otimes K^5_{1}(G^*)\xrightarrow{\eta_\Phi}\bigwedge^2F\otimes K^5_{0}(G^*)\xrightarrow{\bigwedge^5\Phi} \bigwedge^7F\xrightarrow{\bigwedge^1\Phi}
\bigwedge^8F\otimes L^0_{1}G\xrightarrow{\Kos_\Phi} \bigwedge ^9F\otimes L^0_{2}G
\to 0 \notag\\
&\scriptscriptstyle \mathcal C_{\Phi}^{2,5}:\ 0\to\bigwedge^0F\otimes K^5_{1}(G^*)\xrightarrow{\eta_\Phi} \bigwedge^1F\otimes K^5_{0}(G^*)\xrightarrow{\bigwedge^5\Phi}\bigwedge^6F\xrightarrow{\bigwedge^1\Phi} \bigwedge^7F\otimes L^0_{1}G\xrightarrow{\Kos_\Phi} 
\bigwedge^8F\otimes L^0_{2}G\xrightarrow{\Kos_\Phi} \bigwedge ^9F\otimes L^0_{3}G
\to 0 \notag\\
&\scriptscriptstyle \mathcal C_{\Phi}^{3,5}:\ 0\to\bigwedge^0F\otimes K^5_{0}(G^*)\xrightarrow{\bigwedge^5\Phi} \bigwedge^5F\xrightarrow{\bigwedge^1\Phi}\bigwedge^6F\otimes L^0_{1}G\xrightarrow{\Kos_\Phi} \bigwedge^7F\otimes L^0_{2}G\xrightarrow{\Kos_\Phi}
\bigwedge^8F\otimes L^0_{3}G\xrightarrow{\Kos_\Phi} \bigwedge ^9F\otimes L^0_{4}G
\to 0 \notag\\
&\scriptscriptstyle \mathcal C_{\Phi}^{4,5}:\ 0\to\bigwedge^4F\xrightarrow{\bigwedge^1\Phi} \bigwedge^5F\otimes L^0_{1}G\xrightarrow{\Kos_\Phi}\bigwedge^6F\otimes L^0_{2}G\xrightarrow{\Kos_\Phi} \bigwedge^7F\otimes L^0_{3}G\xrightarrow{\Kos_\Phi}
\bigwedge^8F\otimes L^0_{4}G\xrightarrow{\Kos_\Phi} \bigwedge ^9F\otimes L^0_{5}G\to 0 \notag\end{align}\endgroup
\end{example}

\begin{example}\label{5.6}If $\goth f=\goth g+1$ and $i=0$, then $\mathcal C^{0,a}_\Phi\otimes \bigwedge^\goth f F^*$ is $$\textstyle 0\to \bigwedge^0F\otimes \bigwedge^\goth f F^*\otimes K_0^a(G^*)\xrightarrow{\bigwedge^a\Phi}\bigwedge^aF\otimes \bigwedge^\goth f F^* \xrightarrow{\bigwedge^{\goth g+1-a}\Phi}\bigwedge^\goth f F\otimes \bigwedge^\goth f F^*\otimes L_1^{\goth g-a}G\to 0,$$which is naturally isomorphic to
\begin{equation}\label{5.6ex}\textstyle 0\to  \bigwedge^\goth f F^*\otimes \bigwedge^a(G^*)\xrightarrow{d_2} \bigwedge^{\goth f-a} F^* \xrightarrow{d_1}\bigwedge^{\goth f-a}G\to 0,\end{equation}
with $d_2(\omega_{F^*}\otimes \gamma_a)=[(\bigwedge^a\Phi)(\gamma_a)](\omega_{F^*})$ and $d_1(\phi_{\goth f-a})=(\bigwedge^{\goth f-a}\Phi^*)(\phi_{\goth f-a})$, for $\omega_{F^*}\in \bigwedge^\goth f F^*$, $\gamma_a\in \bigwedge^a(G^*)$, and $\phi_{\goth f-a}\in \bigwedge^{\goth f-a}F$. In particular, if $\goth f=4$,  $\goth g=3$, $a=2$, and $\Phi=(\Phi_{i,j})$ is given by a $4\times 3$ matrix, then (\ref{5.6ex}) is
$$0\to R^3\xrightarrow{d_2}R^6\xrightarrow{d_1}R^3 \to 0$$ with
$$d_1=\bmatrix 
\Delta(1,2;1,2)&\Delta(1,3;1,2)&\Delta(1,4;1,2)&\Delta(3,4;1,2)&\Delta(2,4;1,2)&
\Delta(2,3;1,2)\\
\Delta(1,2;1,3)&\Delta(1,3;1,3)&\Delta(1,4;1,3)&\Delta(3,4;1,3)&\Delta(2,4;1,3)&
\Delta(2,3;1,3)\\
\Delta(1,2;2,3)&\Delta(1,3;2,3)&\Delta(1,4;2,3)&\Delta(3,4;2,3)&\Delta(2,4;2,3)&
\Delta(2,3;2,3)\endbmatrix$$
and $$d_2=\bmatrix 
\Delta(3,4;1,2)&\Delta(3,4;1,3)&\Delta(3,4;2,3)\\
-\Delta(2,4;1,2)&-\Delta(2,4;1,3)&-\Delta(2,4;2,3)\\
\Delta(2,3;1,2)&\Delta(2,3;1,3)&\Delta(2,3;2,3)\\
\Delta(1,2;1,2)&\Delta(1,2;1,3)&\Delta(1,2;2,3)\\
-\Delta(1,3;1,2)&-\Delta(1,3;1,3)&-\Delta(1,3;2,3)\\
\Delta(1,4;1,2)&\Delta(1,4;1,3)&\Delta(1,4;2,3)\endbmatrix,$$
where $$\Delta(i,j;k,\ell)=\det\bmatrix \Phi_{i,k}&\Phi_{i,\ell}\\\Phi_{j,k}&\Phi_{j,\ell}\endbmatrix.$$
\end{example}

\begin{observation} \label{new=old} Adopt Data~{\rm\ref{setup}}. Recall the complexes $\{\mathcal C^{i}_\Phi\}$ of Definition~{\rm\ref{FFR}} and the complexes 
$\{\mathcal C^{i,a}_\Phi\}$ of Definition~{\rm\ref{CiaPhi}} and Observation~{\rm\ref{complex}}. 
 Then, for each integer $i$, the complexes 
$$\mathcal C^{i}_\Phi, \quad  \mathcal C^{i,1}_\Phi, \quad\text{and}\quad \mathcal C^{i-1,\goth g}_\Phi\otimes \textstyle \bigwedge^\goth g G$$
are canonically isomorphic.\end{observation}
\begin{proof}Use the formulas (\ref{Ci-spot}) and (\ref{spot-by-spot}), Observation~\ref{2.8}, and the fact that $\bigwedge^\goth gG\otimes \bigwedge^\goth gG^*$ is canonically isomorphic to $R$ to see that the modules
 $$(\mathcal C^{i}_\Phi)_j, \quad  (\mathcal C^{i,1}_\Phi)_j, \quad\text{and}\quad (\mathcal C^{i-1,\goth g}_\Phi)_j\otimes \textstyle \bigwedge^\goth g G$$
are canonically isomorphic for all $j$. These canonical isomorphisms induce the required canonical isomorphisms of complexes. 
\end{proof}

\begin{observation}\label{duality} Adopt Data~{\rm\ref{setup}}. Let and $i$ and $a$ be integers with $1\le a\le \goth g$. Then the complexes 
$$\mathcal C_\Phi^{i,a}\otimes \textstyle \bigwedge^{\goth f}(F^*) \quad\text{and}\quad \mathcal (C_\Phi^{\goth f-\goth g-i-1,\goth g+1-a})^*[-(\goth f-\goth g+1)]$$ of Definition~{\rm\ref{CiaPhi}} and Observation~{\rm\ref{complex}} are canonically isomorphic.
\end{observation}
\begin{Remark} The symbol ``$[-(\goth f-\goth g+1)]$'' refers to a shift in homological degree, see \ref{shift}. In particular, for each integer $j$,
$$((C_\Phi^{\goth f-\goth g-i-1,\goth g+1-a})^*[-(\goth f-\goth g+1)])_j=
 ((C_\Phi^{\goth f-\goth g-i-1,\goth g+1-a})_{\goth f-\goth g+1-j})^*.$$
\end{Remark}
\begin{proof}
Use the formula  (\ref{spot-by-spot}), Observation~\ref{2.8}, and the fact that $\bigwedge^\ell F^*$ is canonically isomorphic to $\bigwedge^{\goth f-\ell}F\otimes \bigwedge^\goth f(F^*)$, for all integers $\ell$ to see that the modules
 $$(\mathcal C^{i,a}_\Phi)_j \otimes \textstyle \bigwedge^\goth f(F^*)\quad\text{and}\quad 
((C_\Phi^{\goth f-\goth g-i-1,\goth g+1-a})_{\goth f-\goth g+1-j})^*
$$
are canonically isomorphic for all $j$.
Indeed, \begingroup\allowdisplaybreaks
\begin{align*}&((C_\Phi^{\goth f-\goth g-i-1,\goth g+1-a})_{\goth f-\goth g+1-j})^*\\
=&\begin{cases}
(\bigwedge^{\goth g-1+j}F\otimes L_{j-i-1}^{a-1}G)^*\,&\text{if $i+2\le j$,}\\
(\bigwedge^{\goth g+i+1-a}F)^*,&\text{if $i+1=j$, and}\\
(\bigwedge^{j}F\otimes K_{i-j}^{\goth g+1-a}(G^*))^*,&\text{if $j\le i$}\end{cases}
\\
\cong&\begin{cases}
\bigwedge^{\goth g-1+j}(F^*)\otimes K_{j-i-2}^{a}(G^*),&\text{if $i+2\le j$,}\\
\bigwedge^{\goth g+i+1-a}(F^*),&\text{if $i+1=j$, and}\\
\bigwedge^{j}(F^*)\otimes L_{i-j+1}^{\goth g-a}G,&\text{if $j\le i$}\end{cases}&&\text{by (\ref{p207})}\\
\cong&\begin{cases}
\bigwedge^{\goth f-\goth g+1-j}F\otimes K_{j-i-2}^{a}(G^*)\otimes \textstyle \bigwedge^\goth f(F^*),&\text{if $i+2\le j$,}\\
\bigwedge^{\goth f-\goth g-i-1+a}F\otimes \textstyle \bigwedge^\goth f(F^*),&\text{if $i+1=j$, and}\\
\bigwedge^{\goth f-j}F\otimes L_{i-j+1}^{\goth g-a}G\otimes \textstyle \bigwedge^\goth f(F^*),&\text{if $j\le i$}\end{cases}\\
=&(\mathcal C^{i,a}_\Phi)_j\otimes \textstyle\bigwedge^{\goth f}(F^*) .
\end{align*}\endgroup
The complex $\mathcal C_\Phi^{i,a}$ is obtained by patching together two well-known complexes. The duality among the pieces is well understood. We focus on the duality at the patch:
\small{$$\xymatrix {(\mathcal C_\Phi^{i,a})_{i+2}\otimes \bigwedge^\goth f(F^*)\ar[r]^{d}\ar[d]&(\mathcal C_\Phi^{i,a})_{i+1}\otimes \bigwedge^\goth f(F^*)\ar[r]^{d}\ar[d]&
(\mathcal C_\Phi^{i,a})_{i}\otimes \bigwedge^\goth f(F^*)\ar[d]\\
((\mathcal C_\Phi^{\goth f-\goth g-i-1,\goth g+1-a})_{\goth f-\goth g-1-i})^*\ar[r]^{d^*}&((\mathcal C_\Phi^{\goth f-\goth g-i-1,\goth g+1-a})_{\goth f-\goth g-i})^*\ar[r]^{d^*}&
((\mathcal C_\Phi^{\goth f-\goth g-i-1,\goth g+1-a})_{\goth f-\goth g+1-i})^*,}$$}
which is the same as $$\scriptstyle\xymatrix {\scriptstyle\bigwedge^{\goth f-\goth g-i-1}F\otimes K^a_0(G^*)\otimes \bigwedge^\goth f(F^*)\ar[r]^(.55){d}\ar[d]&\scriptstyle\bigwedge^{\goth f-\goth g+a-i-1}F\otimes \bigwedge^\goth f(F^*)\ar[r]^{d}\ar[d]&\scriptstyle
\bigwedge^{\goth f-i}F\otimes L_1^{\goth g-a}G\otimes \bigwedge^\goth f(F^*)\ar[d]\\\scriptstyle
(\bigwedge^{\goth g+1+i}F\otimes L_1^{a-1}G)^*\ar[r]^(.55){d^*}&\scriptstyle(\bigwedge^{\goth g+i+1-a}F)^*
\ar[r]^{d^*}&\scriptstyle
(\bigwedge^iF\otimes K_0^{\goth g+1-a}(G^*))^*.}$$ The vertical maps are comprised of the canonical isomorphisms $\bigwedge^\ell F\otimes \bigwedge^\goth fF^*\cong \bigwedge ^{\goth f-\ell} F^*$ and the isomorphisms induced by the perfect pairing of \ref{p207}. 
There is no difficulty in checking that the diagram commutes up to sign.
\end{proof}

Our conventions concerning the length of a complex are given in \ref{length}.

\begin{observation}\label{3.7}Adopt Data~{\rm\ref{setup}}. Let and $i$ and $a$ be integers with $1\le a\le \goth g$. Recall the complexes $\{\mathcal C_\Phi^{i,a}\}$ of Definition~{\rm\ref{CiaPhi}} and Observation~{\rm\ref{complex}}. 
\begin{enumerate}[\rm(a)]
\item\label{3.7.a} If $-1\le i$, then $(\mathcal C_\Phi^{i,a})_j=0$ for $j\le -1$ and  $\length(\mathcal C^{i,a}_\Phi)\le \goth f$.
\item If $i\le \goth f-\goth g$ and $\goth f-\goth g+2\le j$, then $(\mathcal C_\Phi^{i,a})_j=0$.
\item\label{3.7.c} Assume
$$-1\le i\le \goth f-\goth g,\text{ or }(i,a)=(\goth f-\goth g+1,1),\text{ or }(i,a)=(-2,\goth g).$$
Then $(\mathcal C_\Phi^{i,a})_j=0$ for $j\le -1$ and for $\goth f-\goth g+2\le j$; in particular, $\mathcal C_\Phi^{i,a}$ 
 is a complex of length at most $\goth f-\goth g+1$.\end{enumerate}
\end{observation}
\begin{proof} Use (\ref{spot-by-spot}). If $j\le -1\le i$, then the  the $\bigwedge^{\bullet} F$ contribution to $(\mathcal C_\Phi^{i,a})_j$ is 
$\bigwedge^{\goth f-j}F=0$. If $i\le \goth f-\goth g$ and $\goth f-\goth g+2\le j$, then the $\bigwedge^{\bullet} F$ contribution to $(\mathcal C_\Phi^{i,a})_j$ is 
$\bigwedge^{\goth f-\goth g+1-j}F=0$. The other assertions are checked in a similar manner.
\end{proof}

\begin{observation}\label{H0}Adopt Data~{\rm\ref{setup}}. Let and $i$ and $a$ be integers with $1\le a\le \goth g$. Recall the complexes $\{\mathcal C_\Phi^{i,a}\}$ of Definition~{\rm\ref{CiaPhi}} and Observation~{\rm\ref{complex}}. Then
$$\HH_0(\mathcal C^{i,a}_\Phi)=\begin{cases}
\frac{\bigwedge^{\goth f-\goth g+a}F}{\im (\bigwedge^a\Phi)\wedge \bigwedge^{\goth f-\goth g}F},&\text{if $i=-1$,}\vspace{5pt}\\
\coker (\bigwedge^{\goth g-a+1}\Phi^*),&\text{if $i=0$, and}\vspace{5pt}\\
\frac{L_{i+1}^{\goth g-a}G}{\Phi^*(F)\cdot L_{i}^{\goth g-a}G},&\text{if $1\le i$.}\end{cases}$$\end{observation}
\begin{proof} Fix a basis element $\omega_{F^*}$ of $\bigwedge^\goth fF^*$ and, for each $j$, let 
$\sigma:\bigwedge^{j}F\to \bigwedge^{\goth f-j}F^*$ be the non-canonical isomorphism which sends $f_{j}$ to $f_{j}(\omega_{F^*})$.

For $i=0$, consider the commutative square
\begin{equation}\label{pict1}\xymatrix{\mathcal (C_\Phi^{0,a})_1=\bigwedge^{\goth f-\goth g+a-1}F\ar[r]^{d_1}\ar[d]_{\cong}^{\sigma}&
 \mathcal (C_\Phi^{0,a})_0=\bigwedge^{\goth f}F\otimes L_{1}^{\goth g-a} G\ar[d]_{\cong}^{\sigma\otimes \text{(\ref{2.8}.\ref{2.8.g})}}\ar[r]&\HH_0(C_\Phi^{0,a})\ar[r]&0\\
\bigwedge^{\goth g+1-a}F^* \ar[r]^{\bigwedge^{\goth g+1-a}\Phi^*}&
  \bigwedge^{\goth g-a+1} G.}\end{equation}

For $1\le i$,  consider the commutative square
\begin{equation}\label{pict2}\xymatrix{\mathcal (C_\Phi^{i,a})_1=\bigwedge^{\goth f-1}F\otimes L_i^{\goth g-a} G\ar[r]^{d_1}\ar[d]_{\cong}^{\sigma_1\otimes 1}&
 \mathcal (C_\Phi^{i,a})_0=\bigwedge^{\goth f}F\otimes L_{i+1}^{\goth g-a} G\ar[d]_{\cong}^{\sigma_0\otimes 1}\ar[r]&\HH_0(C_\Phi^{i,a})\ar[r]&0\\
F^*\otimes L_i^{\goth g-a} G\ar[r]^{\delta}&
  L_{i+1}^{\goth g-a} G,}\end{equation}
where 
$\delta(\phi\otimes \sum_\ell A_\ell\otimes B_\ell)=\sum_\ell A_\ell \otimes \Phi^*(\phi)\cdot B_\ell$ for $\phi\in F^*$, $A_\ell\in \bigwedge^{\goth g-a}G$, $B_{\ell}\in \Sym_{i}G$  and 
$$\textstyle \sum_\ell A_\ell \otimes B_\ell\in L_i^{\goth g-a}G\subseteq \bigwedge^{\goth g-a}G\otimes \Sym_{i}G.$$
In both diagrams, (\ref{pict1}) and (\ref{pict2}), the top line is exact by \ref{3.7}.\ref{3.7.a}. For $i=-1$, 
the sequence
$$\textstyle\mathcal (C_\Phi^{-1,a})_1=\bigwedge^{\goth f-\goth g}F\otimes 
\bigwedge^aG^*\xrightarrow{d_1}
\mathcal (C_\Phi^{-1,a})_0=\bigwedge^{\goth f-\goth g+a}F\to\HH_0(C_\Phi^{-1,a})\to 0$$ is exact and $d_1$ sends $f_{\goth f-\goth g}\otimes \gamma_a$ to  $f_{\goth f-\goth g}\wedge (\bigwedge^a\Phi) (\gamma_a)$. 
\end{proof}

\begin{observation}\label{ann(H0)}Adopt Data~{\rm\ref{setup}}.  Let and $i$ and $a$ be integers with $1\le a\le \goth g$. Recall the complexes $\{\mathcal C_\Phi^{i,a}\}$ of Definition~{\rm\ref{CiaPhi}} and Observation~{\rm\ref{complex}}. Then $I_\goth g(\Phi)$ annihilates 
$\HH_0(\mathcal C^{i,a}_\Phi)$ for all $i$ with $-1\le i$.\end{observation}

\begin{proof}We show that $I_\goth g(\Phi)\cdot (\mathcal C^{i,a}_{\Phi})_0\subseteq \im d_1$ for each relevant $i$. Consider the arbitrary element $r=[(\bigwedge^\goth g \Phi)(\omega_{G^*})](\phi_{\goth g})$   of $I_\goth g(\Phi)$, where
$\omega_{G^*}\in \bigwedge^\goth g G^*$ and  $\phi_{\goth g}\in \bigwedge^\goth g F^*$.

We  prove the assertion for  $i=-1$ by showing that $r\cdot 
\bigwedge^{\goth f-\goth g+a}F\subseteq \im (\bigwedge^a\Phi)\wedge \bigwedge^{\goth f-\goth g}F$. Let $f_{\goth f-\goth g+a}\in \bigwedge^{\goth f-\goth g+a}F$. Apply Proposition~\ref{use-in-5.10} in order to write $r\cdot f_{\goth f-\goth g+a}$, which is equal to
$$ f_{\goth f-\goth g+a}\wedge \phi_\goth g\Big(({\textstyle\bigwedge}^\goth g \Phi)(\omega_{G^*})\Big),$$
as a sum of elements of the form
$$f\wedge \phi\Big(({\textstyle\bigwedge}^\goth g \Phi)(\omega_{G^*})\Big),$$for  homogeneous elements  $f\in \bigwedge^\bullet F$ and $\phi\in \bigwedge^\bullet F$ with$$\deg \phi\le \goth g+\goth f-(\goth f-\goth g+a)-\goth g=\goth g-a.$$ Apply Proposition~\ref{A3}.\ref{A3.d} to see that 
$$f\wedge \phi\Big(({\textstyle\bigwedge}^\goth g \Phi)(\omega_{G^*})\Big)
= f\wedge ({\textstyle\bigwedge}^{\goth g-\deg \phi}  \Phi)
\Big(\Big[({\textstyle\bigwedge}^{\deg \phi} \Phi^*)(\phi)\Big](\omega_{G^*})\Big).
$$
This completes the proof when $i=-1$ because $a\le \goth g-\deg \phi$.

Consider $i=0$. In light of (\ref{pict1}), it suffices to show that \begin{equation}\label{sts}\textstyle 
r(\bigwedge^ {\goth g+1-a}G)\subseteq
(\bigwedge ^{\goth g+1-a}\Phi^*)(\bigwedge^{\goth g+1-a}F^*) .\end{equation}
Let $g_{\goth g+1-a}$ be an element of $\bigwedge^ {\goth g+1-a}G$. 
Observe that $$\textstyle\big[(\bigwedge^{a-1}\Phi)
[g_{\goth g+1-a}(\omega_{G^*})]\big] (\phi_{\goth g})$$ is an element of $\bigwedge^{\goth g+1-a}F^*$ and $\bigwedge^{\goth g+1-a}\Phi^*$ carries this element to 
\begingroup\allowdisplaybreaks\begin{align*} &\textstyle(\bigwedge^{\goth g+1-a}\Phi^*)
\Big(\big[(\bigwedge^{a-1}\Phi)
[g_{\goth g+1-a}(\omega_{G^*})]\big] (\phi_{\goth g})\Big)
 \\
=&\textstyle
[g_{\goth g+1-a}(\omega_{G^*})]
\Big((\bigwedge^{\goth g}\Phi^*)(\phi_{\goth g})\Big)
&&\text{by \ref{A3}.\ref{A3.d}}\\
=&\textstyle
g_{\goth g+1-a}\wedge \omega_{G^*}
\Big((\bigwedge^{\goth g}\Phi^*)(\phi_{\goth g})\Big)
&&\text{by \ref{A3}.\ref{A3.c}}\\
=&r\cdot g_{\goth g+1-a}.
\end{align*}   \endgroup
The claim (\ref{sts}) has been established. This completes the proof when $i=0$.

One further consequence of (\ref{sts}) is that $rG\subseteq \Phi^*(F^*)$. (Take $a$ to be $\goth g$ to obtain this conclusion.) 

We prove the assertion for $1\le i$ by showing that 
$$r L_{i+1}^{\goth g-a}G\subseteq \delta(F^*\otimes L_i^{\goth g-a}G),$$ in the notation of (\ref{pict2}). Observe that the exact sequence (\ref{KOS})
yields $$\textstyle  L_{i+1}^{\goth g-a}G=\kappa(\bigwedge^{\goth g-a+1}G\otimes \Sym_{i}G)$$since $0<\goth g-a+i+1$. On the other hand, 
$$\textstyle \kappa(\bigwedge^{\goth g-a+1}G\otimes \Sym_{i-1}G)\subseteq L_i^{\goth g-a}G.$$It follows that 
\begingroup\allowdisplaybreaks\begin{align*}\textstyle rL_{i+1}^{\goth g-a}G&\textstyle \subseteq rG\cdot \Big( \kappa(\bigwedge^{\goth g-a+1}G\otimes \Sym_{i-1}G)\Big)\subseteq rG\cdot \Big( L_i^{\goth g-a}G\Big)\\&\textstyle \subseteq \Phi^*(F^*)\cdot L_i^{\goth g-a}G =\delta (F^*\otimes L_i^{\goth g-a}G),\end{align*}\endgroup where the multiplication ``$\cdot$'' means multiply into the the symmetric algebra factor. The proof is complete for $1\le i$.
\end{proof}

\bigskip 

\section{The acyclicity of $\mathcal C^{i,a}_\Phi$.}

\bigskip
Theorem~\ref{main} is the main result of the paper. It asserts if $\Phi$ is sufficiently general, then $\mathcal C^{i,a}_\Phi$ is a resolution of $\HH_0(\mathcal C^{i,a}_\Phi)$  and $\HH_0(\mathcal C^{i,a}_\Phi)$ is a torsion-free $(R/I_\goth g(\Phi))$-module, for $-1\le i$ and $1\le a\le \goth g$. 
The depth-sensitivity assertion, Corollary~\ref{emain.cc},  was promised in (\ref{DS}), and is in fact our main motivation for writing the paper.
 Corollary~\ref{max-CM} records the fact, promised in (\ref{promise}), that in the generic situation, with
$-1\le i\le \goth f-\goth g$, then 
 $\HH_0(\mathcal C_\Phi^{i,a})$ is a maximal Cohen-Macaulay module of rank $\binom{\goth g-1}{a-1}$ over the determinantal ring $R/I_\goth g(\Phi)$. 
Recall, from \ref{H0}, that 
$$\HH_0(\mathcal C^{i,a}_\Phi)=\begin{cases}
\frac{\bigwedge^{\goth f-\goth g+a}F}{\im (\bigwedge^a\Phi)\wedge \bigwedge^{\goth f-\goth g}F},&\text{if $i=-1$,}\vspace{5pt}\\
\coker (\bigwedge^{\goth g-a+1}\Phi^*),&\text{if $i=0$, and}\vspace{5pt}\\
\frac{L_{i+1}^{\goth g-a}G}{\Phi^*(F)\cdot L_{i}^{\goth g-a}G},&\text{if $1\le i$.}\end{cases}$$
Theorem~\ref{main}
follows readily from Lemma~\ref{key} by way of the acyclicity lemma.
If  $I_1(\Phi)=R$, then it is shown in 
Lemma~\ref{key} that  
$$\mathcal C_\Phi^{i,a} \quad \text{and} \quad \mathcal C_{\Phi'}^{i,a}\oplus \mathcal C_{\Phi'}^{i,a-1}$$ have isomorphic homology for some smaller 
$R$-module homomorphism $\Phi'$.

\begin{lemma}\label{key} Adopt Data~{\rm\ref{8data-8-8}} and Notation~{\rm\ref{notation-9}} with $\goth g\le \goth f$ and $\Phi''$ an isomorphism. Let $a$ and $i$ be integers with   $1\le a\le \goth g-1$. Recall the complex $\mathcal C^{i,a}_\Phi$ of Definition~{\rm\ref{CiaPhi}} and Observation~{\rm\ref{complex}}. 
Then there 
there exists a canonical complex $\mathcal D$ of free $R$-modules and 
canonical short exact sequences
\begin{equation}\label{ses1}0\to \Tot\Big(\mathbb L^{\goth f-i-1,\goth g-a}_{\pi'\circ \Phi}\otimes \mathbb A(\Phi'')\Big)\xrightarrow{\incl^\dagger}\mathcal C^{i,a}_\Phi\xrightarrow{\quot^\dagger} \mathcal D\to 0\end{equation}
and
\begin{equation}\label{ses2}0\to \mathcal C^{i,a}_{\Phi'}\oplus\big(\mathcal C^{i,a-1}_{\Phi'}\otimes {G''}^*\big) \xrightarrow{\incl^\ddag}\mathcal D\xrightarrow{\quot^\ddag}
\Tot\Big(\mathbb K^{\goth f-\goth g-i-2,a}_{\pi'\circ \Phi}[-i-2]\otimes \mathbb B(\Phi'')\Big)\to 0.\end{equation}
In particular,   there are  canonical 
isomorphisms
\begin{equation}\label{iso-3}\HH_j(\mathcal C_{\Phi}^{i,a})\cong \HH_j(\mathcal C_{\Phi'}^{i,a})\oplus \Big(\HH_j(\mathcal C_{\Phi'}^{i,a-1})\otimes {G''}^*\Big)\end{equation} for all integers $j$. 
\end{lemma}
\begin{remarks-no-advance}\label{6.2.1}\begin{enumerate}
[\rm(a)]
\item\label{6.2.1.b} Lemma~\ref{key} does apply when $a=1$  provided one interprets $\mathcal C^{i,a-1}_{\Phi'}$  using Remark~\ref{July-6}. The complex $\mathcal C^{i,0}_{\Phi'}$  of Remark~\ref{July-6} is split exact; so the ultimate conclusion of Lemma \ref{key} when $a=1$   is  
$$\HH_j(\mathcal C_{\Phi}^{i,1})\cong \HH_j(\mathcal C_{\Phi'}^{i,1})
,$$ for all integers $j$.  In light of Observation~\ref{new=old}, this conclusion is well-known.

\item\label{6.2.1.c}  Lemma~\ref{key} is false when $a=\goth g$, even if one interprets $\mathcal C^{i,\goth g}_{\Phi'}$  using Remark~\ref{July-6}.
The correct statement is that 
$$\HH_j(\mathcal C_{\Phi}^{i,\goth g})\cong \HH_j(\mathcal C_{\Phi'}^{i,\goth g-1})\quad\text{and}\quad\HH_j(\mathcal C_{\Phi}^{i,\goth g}) \not \cong\HH_j(\mathcal C_{\Phi'}^{i,\goth g})\oplus \HH_j(\mathcal C_{\Phi'}^{i,\goth g-1}).$$ 
 Indeed, 
\begingroup\allowdisplaybreaks\begin{align*}\HH_j(\mathcal C_{\Phi}^{i,\goth g})&
\cong \HH_j(\mathcal C_{\Phi}^{i+1,1})&&\text{by Observation~\ref{new=old}}\\
&\cong \HH_j(\mathcal C_{\Phi'}^{i+1,1})&&\text{by (\ref{6.2.1.b})}\\
&\cong \HH_j(\mathcal C_{\Phi'}^{i,\goth g-1})&&\text{by Observation~\ref{new=old}}.
\end{align*}\endgroup On the other hand, the complex $\mathcal C_{\Phi'}^{i,\goth g}$ has non-zero homology because the complex is  non-zero in exactly one position; see Remark~\ref{July-6}.
\end{enumerate}
\end{remarks-no-advance}

\begin{proof-of-key} 
We know from Definition~\ref{CiaPhi} that 
$\mathcal C^{i,a}_\Phi$ is the complex 
$$\textstyle 0\to \mathbb K^{\goth f-\goth g-i-1,a}_\Phi[-i-2]\xrightarrow{d} \bigwedge^{\goth f-\goth g+a-i-1}F \xrightarrow{d} \mathbb L^{\goth f-i-1,\goth g-a}_\Phi\to 0,$$with $[\mathcal C^{i,a}_\Phi]_{i+1}=\bigwedge^{f-g+a-i-1}F$.
The parameter $\goth g-a$ is positive; so, we 
  know from Proposition~\ref{87.5}.\ref{87.5.b} that 
$$0\to \Tot\Big(\mathbb L^{\goth f-i-1,\goth g-a}_{\pi'\circ \Phi}\otimes \mathbb A(\Phi'')\Big)\xrightarrow{\incl^\dagger}\mathbb L^{\goth f-i-1,\goth g-a}_\Phi\xrightarrow{\quot^\dagger} 
{\begin{matrix}\mathbb L^{\goth f-i-2,\goth g-a}_{\Phi'}\otimes F''\\\oplus\\   \mathbb L^{\goth f-i-2,\goth g-a-1}_{\Phi'}\otimes {F''} \otimes G''\end{matrix}}\to 0$$ is 
a short exact sequence of complexes.
 It follows that \begin{equation}\label{recent}\Tot\Big(\mathbb L^{\goth f-i-1,\goth g-a}_{\pi'\circ \Phi}\otimes \mathbb A(\Phi'')\Big)\xrightarrow{\incl^\dagger}\mathcal C^{i,a}_\Phi\end{equation} is an injection of complexes. Let $\mathcal D$ be the cokernel of (\ref{recent}). Observe that (\ref{ses1}) is a short exact sequence of complexes and that $\mathcal D$ is isomorphic  to 
$$\textstyle 0\to \mathbb K^{\goth f-\goth g-i-1,a}_\Phi[-i-2]\xrightarrow{d} \bigwedge^{\goth f-\goth g+a-i-1}F  \xrightarrow{\quot^\dagger\circ d} 
{\begin{matrix}\mathbb L^{\goth f-i-2,\goth g-a}_{\Phi'}\otimes F''\\\oplus\\   \mathbb L^{\goth f-i-2,\goth g-a-1}_{\Phi'}\otimes {F''} \otimes G''\end{matrix}}\to 0.$$
The parameter $a$ is positive; and therefore, we also know from Proposition~\ref{87.5}.\ref{87.5.a} that 
\begin{align*}0\to{\begin{matrix}\mathbb K^{\goth f-\goth g-i-1,a}_{\Phi'}\\\oplus\\  (\mathbb K^{\goth f-\goth g-i-1,a-1}_{\Phi'}\otimes {G''}^*)\end{matrix}}\xrightarrow{\incl^\ddag} \mathbb K^{\goth f-\goth g-i-1,a}_\Phi\xrightarrow{\quot^\ddag}
\Tot\Big(\mathbb K^{\goth f-\goth g-i-2,a}_{\pi'\circ \Phi} \otimes 
\mathbb B(\Phi'')
\Big)\to 0
\end{align*}is a short exact sequence of complexes. It follows that 
$$\mathcal D \xrightarrow{\quot^\ddag}\Tot\Big(\mathbb K^{\goth f-\goth g-i-2,a}_{\pi'\circ \Phi} \otimes 
\mathbb B(\Phi'')
\Big)[-i-2]$$ is a surjection of complexes with kernel isomorphic  to the complex:
{\footnotesize\begin{align}\label{8:11}\textstyle 0\to
{\begin{matrix}\mathbb K^{\goth f-\goth g-i-1,a}_{\Phi'}[-i-2]\\\oplus\\  \mathbb K^{\goth f-\goth g-i-1,a-1}_{\Phi'}[-i-2]\otimes {G''}^*\end{matrix}}
\xrightarrow{d\circ \incl^{\ddag}} \bigwedge^{\goth f-\goth g+a-i-1}F  \xrightarrow{\quot^\dagger\circ d} 
{\begin{matrix}\mathbb L^{\goth f-i-2,\goth g-a}_{\Phi'}\otimes F''\\\oplus\\   \mathbb L^{\goth f-i-2,\goth g-a-1}_{\Phi'}\otimes {F''} \otimes G''\end{matrix}}\to 0.
\end{align}}Thus, 
\begin{equation}\label{Thus}0\to (\ref{8:11}) \xrightarrow{\incl^\ddag}\mathcal D \xrightarrow{\quot^\ddag}\Tot\Big(\mathbb K^{\goth f-\goth g-i-2,a}_{\pi'\circ \Phi} \otimes 
\mathbb B(\Phi'')
\Big)[-i-2]\to 0\end{equation} is a short exact sequence  of complexes. Observe that 
\begin{equation}\label{Observe} \xymatrix{\mathcal C^{i,a}_{\Phi'} \oplus(\mathcal C^{i,a-1}_{\Phi'}\otimes {G''}^*)
\ar[d]
\\
{\rm(\ref{8:11}),}
}\end{equation}given by
{\small$$\xymatrix{
0\to
{\begin{matrix}\mathbb K^{\goth f-\goth g-i-1,a}_{\Phi'}[-i-2]\\\oplus\\  \mathbb K^{\goth f-\goth g-i-1,a-1}_{\Phi'}[-i-2]\otimes {G''}^*\end{matrix}}
\ar[r]^(.6){\left[\smallmatrix d&0\\0&d \endsmallmatrix \right]}
\ar[d]^{\left[\smallmatrix \id&0\\0&\id\endsmallmatrix \right]}& 
{\begin{matrix} \bigwedge^{\goth f-\goth g+a-i-1}F'\\\oplus\\
\bigwedge^{\goth f-\goth g+a-i-2}F'\otimes {G''}^*\end{matrix}}  \ar[r]^(.4){\left[\smallmatrix d&0\\0&d \endsmallmatrix \right]}\ar[d]^{\left[\smallmatrix \incl&\incl^\dagger\circ (1\otimes \Phi'')\endsmallmatrix\right]}& 
{\begin{matrix}\mathbb L^{\goth f-i-2,\goth g-a-1}_{\Phi'}\\\oplus\\   \mathbb L^{\goth f-i-2,\goth g-a}_{\Phi'} \otimes {G''}^*\end{matrix}}\to 0\ar[d]^{\left[
\smallmatrix 0&1\otimes \Phi''\\1\otimes \Kos_{\Phi''}&0 \endsmallmatrix\right]}\\
0\to
{\begin{matrix}\mathbb K^{\goth f-\goth g-i-1,a}_{\Phi'}[-i-2]\\\oplus\\  \mathbb K^{\goth f-\goth g-i-1,a-1}_{\Phi'}[-i-2]\otimes {G''}^*\end{matrix}}
\ar[r]^(.6){d\circ \incl^{\ddag}}& \bigwedge^{\goth f-\goth g+a-i-1}F  \ar[r]^(.4){\quot^\dagger\circ d}& 
{\begin{matrix}\mathbb L^{\goth f-i-2,\goth g-a}_{\Phi'}\otimes F''\\\oplus\\   \mathbb L^{\goth f-i-2,\goth g-a-1}_{\Phi'}\otimes {F''} \otimes G''\end{matrix}}\to 0,
}$$} 

\noindent is an isomorphism of complexes. (The complexes $\mathcal C^{i,a}_{\Phi'}$ and $\mathcal C^{i,a-1}_{\Phi'}$ have been read from Definition~\ref{CiaPhi}.) Combine (\ref{Thus}) and (\ref{Observe}) in order to see that (\ref{ses2}) is a short exact sequence of complexes. 
 The isomorphisms of (\ref{iso-3}) follow immediately from the short exact sequences (\ref{ses1}) and (\ref{ses2}) because the two total complexes have all homology equal to zero since $\Phi''$ is an isomorphism.
\qed\end{proof-of-key}

In Lemma~\ref{6.3} we iterate Lemma~\ref{key}. The basic set-up is similar to, but not the same as, the set-up of Data~\ref{8data-8-8}. 

\begin{data}\label{data6} 
Let $R$ be a commutative Noetherian ring, $F$ and $G$ be free $R$-modules of rank $\goth f$ and $\goth g$, respectively, with $\goth g\le \goth f$, and $\Phi:G^*\to F$ be an $R$-module homomorphism.
Decompose $F$ and $G$ as
$$F=F'\oplus F'' \quad\text{and}\quad G=G'\oplus G'',$$ where $F'$, $F''$, $G'$ and $G''$  are free $R$-modules and $\rank F''=\rank G''=r$ for some  integer $r$ with $1\le r\le \goth g-1$ and 
and  let $$F^*={F'}^*\oplus {F''}^*\quad\text{and}\quad  G^*={G'}^*\oplus {G''}^*$$ be the corresponding decompositions of $F^*$ and $G^*$.
Assume that
 \begin{equation}\label{Phi'}\Phi=\bmatrix \Phi_r'&0\\0&\Phi_r''\endbmatrix,\end{equation} where $\Phi_r':{G'}^*\to F'$ is an $R$-module homomorphism and $\Phi_r'':{G''}^*\to F''$ is an $R$-module isomorphism.
\end{data}

\begin{lemma}\label{6.3} Adopt Data~{\rm\ref{data6}}. Let $i$ and $a$ be integers with $1\le a\le \goth g$. Recall the complex $\mathcal C^{i,a}_\Phi$ of Definition~{\rm\ref{CiaPhi}} and Observation~{\rm\ref{complex}}.
\begin{enumerate}[\rm(a)]
\item \label{6.3.a}If $1\le r\le \goth g-1$, then
\begin{equation}\label{6.3gts}\HH_j(\mathcal C^{i,a}_\Phi)\cong 
\bigoplus\limits_{\beta=1}^{\goth g-r}
\HH_j(\mathcal C^{i,\beta}_{\Phi'_{r}})^{\binom{r}{a-\beta}}.\end{equation}
\item\label{6.3.b} If $r=\goth g-1$, then the following statements hold$:$ \begin{enumerate}[\rm(i)]
\item \label{6.3.b.i}
$\HH_j(\mathcal C^{i,a}_\Phi)\cong \HH_j(\mathcal C^{i,1}_{\Phi'_{\goth g-1}})^{\binom{\goth g-1}{a-1}}$, 
\item\label{6.3.b.ii}if 
$I_{\goth g}(\Phi)=R$, then 
the complex $\mathcal C^{i,a}_\Phi$ is split exact$;$ 
and,  \item \label{6.3.b.iii}if  
$ I_{\goth g}(\Phi)$ is a proper ideal  of grade at least 
$\goth f-\goth g+1$,
then $\mathcal C^{i,a}_\Phi$ is a resolution of $(R/I_{\goth g}(\Phi))^{\binom{\goth g-1}{a-1}}$.\end{enumerate}\end{enumerate}
\end{lemma}
\begin{proof}(\ref{6.3.a}). We are given $$\Phi=\bmatrix \Phi_{r}'&0\\0&\Phi_{r}''\endbmatrix,$$ where $\Phi_{r}''$ is an isomorphism of free modules of rank $r$. We may rearrange the data so that
$$\Phi=\bmatrix \Phi_{r}'&0&0\\0&\Phi'''&0\\0&0&\Phi'' \endbmatrix,$$
where $\Phi''$ is an isomorphism of free modules of rank one and  $\Phi'''$ is an isomorphism of free modules of rank $r-1$. Let $$\Phi'=\bmatrix \Phi_{r}'&0\\0&\Phi'''\endbmatrix.$$ Apply Lemma~\ref{key} and Remarks \ref{6.2.1}.\ref{6.2.1.b} and \ref{6.2.1}.\ref{6.2.1.c} to obtain
\begin{equation}\label{6.2.1*}\HH_j(\mathcal C^{i,a}_\Phi)\cong
\chi(a\le \goth g-1)\HH_j(\mathcal C^{i,a}_{\Phi'})
\oplus \chi(2\le a)
\HH_j(\mathcal C^{i,a-1}_{\Phi'}),\end{equation}
where $\chi$ is described in \ref{chi}.
Notice that (\ref{6.2.1*}) agrees with (\ref{6.3gts}) when $r=1$ because
$$\binom{1}{a-\beta}=\begin{cases} 
1,&\text{if $\beta=a$ and $1\le a\le \goth g-1$,}\\
1,&\text{if $\beta=a-1$ and $2\le a\le \goth g$},\\
0,&\text{if $\beta\not\in \{a-1,a\}$ and $1\le \beta\le \goth g-1$.}\end{cases}$$
Induction on  $r$ applied to  (\ref{6.2.1*}) now yields
$$\HH_j(\mathcal C^{i,a}_\Phi)\cong
\chi(a\le \goth g-1)\bigoplus\limits_{\beta=1}^{(\goth g-1)-(r-1)}
\HH_j(\mathcal C^{i,\beta}_{\Phi'_{r}})^{\binom{r-1}{a-\beta}}
\oplus \chi(2\le a)
\bigoplus\limits_{\beta=1}^{(\goth g-1)-(r-1)}
\HH_j(\mathcal C^{i,\beta}_{\Phi'_{r}})^{\binom{r-1}{a-1-\beta}}
.$$The constraints $\chi(a\le \goth g-1)$ and $\chi(2\le a)$ are redundant. Indeed, if $\goth g\le a$, then 
$$\textstyle\beta\le \goth g-r\implies r\le \goth g-\beta\le a-\beta\implies \binom{r-1}{a-\beta}=0,$$ and if $a\le 1$, then 
$$\textstyle 1\le \beta \implies a-1-\beta \le -1\implies \binom{r-1}{a-1-\beta}=0.$$

Thus,  \begingroup\allowdisplaybreaks\begin{align*}\HH_j(\mathcal C^{i,a}_\Phi)&\cong
\bigoplus\limits_{\beta=1}^{\goth g-r}
\HH_j(\mathcal C^{i,\beta}_{\Phi'_{r}})^{\binom{r-1}{a-\beta}} \oplus
\bigoplus\limits_{\beta=1}^{\goth g-r}
\HH_j(\mathcal C^{i,\beta}_{\Phi'_{r}})^{\binom{r-1}{a-1-\beta}}\\
&\cong
\bigoplus\limits_{\beta=1}^{\goth g-r}
\HH_j(\mathcal C^{i,\beta}_{\Phi'_{r}})^{\binom{r}{a-\beta}}.
\end{align*}\endgroup

\medskip\noindent (\ref{6.3.b}). Assertion (\ref{6.3.b.i}) is a special case of (\ref{6.3.a}).  
Recall that $\Phi'_{\goth g-1}:{G'}^*\to F'$ is a homomorphism,  $G'$ is a free module  of rank one, and $F'$ is a free module of rank $\goth f-\goth g+1$.  
The complex $\mathcal C^{i,1}_{\Phi'_{g-1}}$ is the Koszul complex on a generating set for ${I_1(\Phi'_{g-1})=I_{\goth g}(\Phi)}$.
  If $I_\goth g(\Phi)=R$, then $\mathcal C^{i,1}_{\Phi'_{g-1}}$ is split exact. This is (\ref{6.3.b.ii}). Otherwise,  the hypothesis $\goth f-\goth g+1\le \grade I_{\goth g}(\Phi)$ ensures that  $\grade I_{\goth g}(\Phi)$ is generated by a regular sequence and therefore 
$\mathcal C^{i,1}_{\Phi'_{g-1}}$ is a resolution of $\HH_0(\mathcal C^{i,1}_{\Phi'_{g-1}})=R/I_{\goth g}(\Phi)$. This is (\ref{6.3.b.iii}).
\end{proof}

Theorem~\ref{main} is the main result of the paper.

\begin{theorem}
\label{main}Adopt Data~{\rm\ref{setup}}. Let  $i$ and $a$ be integers with $1\le a\le \goth g$. Recall the complex $\mathcal C_\Phi^{i,a}$ from Definition~{\rm \ref{CiaPhi}} and Observation~{\rm\ref{complex}}.    
Assume that
$I_{\goth g}(\Phi)$ is a proper ideal of $R$ with $ 
\goth f-\goth g+1\le \grade I_{\goth g}(\Phi)$.

\begin{enumerate} [\rm(a)]
\item\label{2.a} If  $\length(\mathcal C_\Phi^{i,a})=\goth f-\goth g+1$ and $(\mathcal C_\Phi^{i,a})_j=0$ for $j\le -1$, then the following statements hold.

\begin{enumerate}[\rm(i)]
\item\label{emain.a} The complex $\mathcal C^{i,a}_\Phi$ is acyclic.
\item\label{emain.c} The $R$-module $\HH_0(\mathcal C^{i,a}_\Phi)$ is  perfect  of projective dimension $\goth f-\goth g+1$. 
\item\label{emain.d} The $(R/I_\goth g(\Phi))$-module $\HH_{0}(\mathcal C^{i,a}_\Phi)$ is 
torsion-free.  
\item\label{emain.d'} If $\goth f-\goth g+2\le \grade I_{\goth g-1}(\Phi)$, then  the $(R/I_\goth g(\Phi))$-module $\HH_{0}(\mathcal C^{i,a}_\Phi)$ has rank $\binom{\goth g-1}{a-1}$. 
\end{enumerate}

\item\label{emain.b.ii} If \ $-1\le i$ and  $\goth f-t+1\le \grade I_{t}(\Phi)$ for all $t$
with $\goth f+1-\length(\mathcal C^{i,a}_\Phi) \leq  
t 
\leq \goth g-1$, 
then 
\begin{enumerate}[\rm(i)]
\item\label{emain.b.ii.1} the complex $\mathcal C^{i,a}_\Phi$ is acyclic,
\item\label{emain.b.ii.2}  $\HH_{j}({\mathcal C}^{\goth f-\goth g-i-1,\goth g+1-a}_\Phi) = 
\Ext_{R}^{\goth f-\goth g+1-j}(\HH_{0}(\mathcal C^{i,a}_\Phi),R)$ 
for all $j$, and
\item\label{emain.b.ii.3}  $\HH_{j}({\mathcal C}^{\goth f-\goth g-i-1,\goth g+1-a}_\Phi) = 0$ for $1\le j$.\end{enumerate}
\item\label{emain.b.iii}  If \ $-1\le i$, $\goth f-\goth g+2 \le \length (\mathcal C^{i,a}_\Phi)$,  and \begin{equation}\label{hyp.b.iii} \goth f - t + 2\le \grade I_{t}(\Phi),\quad\text{for all
$t$ with $\goth f+1- \length(\mathcal C^{i,a}_\Phi)
\leq t \leq \goth g - 1$},\end{equation} then $\HH_{0}(\mathcal C^{i,a}_\Phi)$ is a torsion-free
$(R/I_{\goth g}(\Phi))$-module of rank $\binom{\goth g-1}{a-1}$.
\end{enumerate}
\end{theorem}

\begin{remarks-no-advance}\begin{enumerate}[\rm(a)]

\item \label{main-se} Recall from Lemma \ref{6.3}.\ref{6.3.b.ii} that if $I_{\goth g}(\Phi) = R$, then  $\mathcal C^{i,a}_\Phi$ 
is split exact.

\item Observation~\ref{3.7} contains elementary facts about the length of the complexes $\mathcal C^{i,a}_\Phi$. In particular, the hypotheses of (\ref{2.a}) are satisfied when   $$-1\le i\le \goth f-\goth g,\quad\text{or}\quad
 (i,a)=(\goth f-\goth g+1,1),\quad\text{or}\quad
 (i,a)=(-2,\goth g).$$
\item In the generic case (when $\Phi$ can be represented by a matrix of variables) all of the grade hypotheses of Theorem~\ref{main} are automatically satisfied because $$\goth f-t+2\le (\goth g-t+1)(\goth f-t+1),\quad \text{whenever $1\le t\le \goth g-1$}.$$

\item The modules $\HH_0(\mathcal C^{i,a}_\Phi)$ are recorded in Observation~\ref{H0}.
\end{enumerate}\end{remarks-no-advance}

\begin{proof}Throughout this proof, let $\mathcal C$ represent $\mathcal C^{i,a}_\Phi$ and $\ell$ and represent $\length(\mathcal C)$.

\medskip\noindent 
{(\ref{emain.a})}.   
The complex $\mathcal C$ has 
the form $$0\to \mathcal C_{\goth f-\goth g+1}\to \cdots \to\mathcal C_1\to \mathcal C_0\to 0
;$$ see Observation~\ref{3.7}. Let $\mathfrak p$ be a prime ideal of $R$ with $\grade \mathfrak p\le \goth f-\goth g$. Observe that 
$$\grade \mathfrak p\le \goth f-\goth g<\goth f-\goth g+1\le \grade I_{\goth g}(\Phi)\le \grade I_{\goth g-1}(\Phi).$$ Thus, there is a $(\goth g-1)\times (\goth g-1)$ minor of $\Phi$ which is a unit in $R_\mathfrak p$ and we may apply Lemma~\ref{6.3}.\ref{6.3.b.iii} in order to conclude that $\mathcal C_\mathfrak p$ is acyclic. The acyclicity criterion (see, for example, \cite[1.4.13]{BH}) now guarantees that $\mathcal C$ is acyclic. 

\medskip\noindent{(\ref{emain.c}).} We know from Observation~\ref{ann(H0)} that $I_\goth g(\Phi)\subseteq \ann(\HH_0(\mathcal C))$; and therefore,
\begingroup\allowdisplaybreaks\begin{align*}
\pd_R\HH_0(\mathcal C)&\le \goth f-\goth g+1&&\text{by \ref{emain.a}}\\
&\le \grade I_\goth g(\Phi)&&\text{by hypothesis}\\
&\le \grade \ann (\HH_0(\mathcal C))\\&\le \pd_R\HH_0(\mathcal C) &&\text{by (\ref{above})}.\end{align*}\endgroup
The proof of (\ref{emain.c}) is complete; see \ref{perfect}, if necessary.

\medskip\noindent Assertion {(\ref{emain.d})} is a consequence of
Proposition~\ref{1.25}.

\medskip\noindent   {(\ref{emain.d'}).} The $R$-module $R/I_g(\Phi)$ is perfect of projective dimension $\goth f-\goth g+1$. (See, for example, \cite[Cor.~5.2]{E61}, \cite[Thm.~1]{N63}. or \cite[2.7]{BV}.) Let $\mathfrak p\in \Spec R$ be an associated prime of $R/I_\goth g(\Phi)$. It follows that $\grade \mathfrak p=\goth f-\goth g+1$ and $I_{\goth g-1}(\Phi)\not\subseteq  \mathfrak p$. Thus, Corollary~\ref{6.3}.\ref{6.3.b.iii} may be applied to $\Phi_\mathfrak p$ in order to conclude that $\HH_0(\mathcal C^{i,a}_\Phi)_{\mathfrak p}\cong \HH_0(R/I_{\goth g}(\Phi))_{\mathfrak p}^{\binom{\goth g-1}{a-1}}$. The proof of {(\ref{emain.d'})} is complete; see \ref{rank}.

\medskip\noindent 
{(\ref{emain.b.ii.1}).}  We induct on $\ell$. The base case, $\ell= \goth f-\goth g+1$, is established in (\ref{emain.a}). We now study the case with
$\goth f-\goth g+2\le \ell$.
 As in the proof of (\ref{emain.a}) we apply the acyclicity criterion and prove that $\mathcal C_\mathfrak p$ is acyclic for all prime ideals $\mathfrak p$ of $R$ with $\grade \mathfrak p<\ell$. Fix such a $\mathfrak p$. The hypotheses of (\ref{emain.b.ii}) with $t=\goth f+1-\ell$ now ensure that $$\grade \mathfrak p<\ell\le \grade I_{\goth f+1-\ell}(\Phi).$$ Thus, there is an $(\goth f+1-\ell)\times (\goth f+1-\ell)$ minor of $\Phi$ which is a unit in in $R_\mathfrak p$ and, after rearrangement,
$$\Phi_{\mathfrak p}=\bmatrix \Phi_{\goth f+1-\ell}'&0\\0&\Phi_{\goth f+1-\ell}''\endbmatrix,$$ as described in (\ref{Phi'}). 
Apply Corollary~\ref{6.3}.\ref{6.3.a} in order to conclude that
\begin{equation}\label{lincom}\begin{array}{l}\HH_j(\mathcal C)_\mathfrak p\text{ is a direct of suitably many copies of modules from the set}\\
\{\HH_j(\mathcal C^{i,\beta}_{\Phi_{\goth f+1-\ell}'})\mid 1\le \beta\le \goth g-(\goth f+1-\ell)\}.\end{array}\end{equation}
Observe that 
$$\goth f-t+1\le \grade I_t(\Phi)\le \grade I_t(\Phi_\mathfrak p)=\grade I_{t-(\goth f+1-\ell)}(\Phi_{\goth f+1-\ell}')$$ for all $t$ with $\goth f+1-\ell\le t\le \goth g-1$. 
Let $$\goth f'=\goth f-(\goth f+1-\ell),\quad \goth g'=\goth g-(\goth f+1-\ell), \quad \text{and}\quad t'=t-(\goth f+1-\ell).$$ We have shown that
$$\goth f'-t'+1\le \grade I_{t'}(\Phi_{\goth f+1-\ell}')\quad \text{for all $t'$ with $0\le t'\le \goth g'-1$.}$$
It follows that the hypotheses of (\ref{emain.b.ii}) apply to $\mathcal C^{i,\beta}_{\Phi_{\goth f+1-\ell}'}$ for each $\beta$ with $1\le \beta\le \goth g'$. On the other hand, 
$$\length(\mathcal C^{i,\beta}_{\Phi_{\goth f+1-\ell}'})\le \rank(\text{the target of }\Phi_{\goth f+1-\ell}')=\goth f'=\goth f-(\goth f+1-\ell)=\ell-1. $$ By induction on $\ell$, each $\mathcal C^{i,\beta}_{\Phi_{\goth f+1-\ell}'}$ is acyclic; and therefore $\mathcal C_\mathfrak p$ is also acyclic by (\ref{lincom}).

\medskip\noindent(\ref{emain.b.ii.2}) and  (\ref{emain.b.ii.3}). Now that we know that $\mathcal C$ is a resolution of $\HH_0(\mathcal C)$,
assertion (\ref{emain.b.ii.2}) can be read from Observation~\ref{duality}; and (\ref{emain.b.ii.3}) follows from (\ref{emain.b.ii.2})
because
the grade $\goth f-\goth g+1$ ideal $I_\goth g(\Phi)$ is contained in the annihilator of $\HH_0(\mathcal C)$.

\medskip\noindent{(\ref{emain.b.iii}).}   
We already know from (\ref{emain.b.ii.1}) that 
\begin{equation}\label{IS}\text{$\mathcal C$ is a free resolution of $\HH_0(\mathcal C)$ of length $\ell$.}\end{equation} 
For each integer $w$, let $F_w$ be the 
ideal in $R$ generated by:
$$  \{ x \in R \mid \pd_{R_x} \HH_0(\mathcal C^{i,a}_\Phi)_x < w \}. $$
\begin{claim-no-advance}\label{claim3} If $\goth f+1-\ell\le t \le \goth g-1$, then $I_t(\Phi)\subseteq F_{\goth f-t+1}$.
\end{claim-no-advance}

\begin{proof-of-claim-3}If $\Delta$ is a $t\times t$ minor of $\Phi$, then one can arrange the data so that $$\Phi_{\Delta}=\bmatrix \Phi'_t&0\\0&\Phi''_t\endbmatrix,$$ where  $\Phi''_t$ is an isomorphism of free $R_\Delta$-modules of rank $t$ as is described in Data~\ref{data6}. Apply Corollary~\ref{6.3}.\ref{6.3.a}
to see that 
\begin{equation}\label{*sum}\HH_j(\mathcal C)_\Delta\cong 
\bigoplus\limits_{\beta=1}^{\goth g-t}
\HH_j(\mathcal C^{i,\beta}_{\Phi'_{t}})^{\binom{t}{a-\beta}},\quad\text{for all $j$}.\end{equation} Apply (\ref{IS}) to see that each $\mathcal C^{i,\beta}_{\Phi'_{t}}$ which actually appears in (\ref{*sum}) (that is, with $\binom t{a-\beta}\neq 0$) is also acyclic. Thus,
$$\pd_{R_\Delta} \HH_0(\mathcal C)_\Delta\le \max\{\length(\mathcal C^{i,\beta}_{\Phi'_{t}})\mid 1\le \beta\le \goth g-t\}\le \rank(\text{the target of $\Phi'_{t}$})=\goth f-t. $$
This completes the proof of Claim~\ref{claim3}. 
\end{proof-of-claim-3}

Combine hypothesis (\ref{hyp.b.iii}) and Claim~\ref{claim3} to see that 
$$\goth f+1-\ell\le t\le \goth g-1 \implies (\goth f-t+1)+1=\goth f-t+2\le \grade I_t(\Phi)\le \grade F_{\goth f-t+1}.$$ Let $w=\goth f-t+1$. We have shown that 
$$\goth f-\goth g+2\le w\le \ell\implies w+1\le \grade F_w.$$ Apply Proposition~\ref{1.25} to conclude that the $R/I_{\goth g}(\Phi)$ module $\HH_0(\mathcal C)$ is torsion-free. 

We re-use the rank calculation of (\ref{emain.d'}). If $\mathfrak p\in \Ass_R(R/I_\goth g(\Phi))$, then
$$\grade \mathfrak p=\goth f-\goth g+1<\goth f-\goth g+3\le \grade I_{\goth g-1}(\Phi);$$
hence, Corollary~\ref{6.3}.\ref{6.3.b.iii} may be applied to $\Phi_\mathfrak p$, as was done in (\ref{emain.d'}), to conclude that $\rank \HH_0(\mathcal C)=\binom{\goth g-1}{a-1}$.\end{proof}

The next result was promised in (\ref{DS}), and is our main motivation for writing the paper.
\begin{corollary}
\label{emain.cc} 
Adopt Data~{\rm\ref{setup}}. Let and $i$ and $a$ be integers with $1\le a\le \goth g$. Recall the complex $\mathcal C_\Phi^{i,a}$ from Definition~{\rm \ref{CiaPhi}} and Observation~{\rm\ref{complex}}.  If $-1\le i\le \goth f-\goth g$, 
then
$$\HH_j(\mathcal C^{i,a}_\Phi)=0 \quad \text{for} \quad \goth f-\goth g+2-\grade I_{\goth g}(\Phi)\le j.$$\end{corollary}
\begin{proof} Assertion~\ref{main}.\ref{2.a} may be applied in the generic case with the ring equal to the polynomial ring $\mathbb Z[\{x_{i,j}\}]$ and the homomorphism given by a matrix of indeterminates.
The present assertion  is a consequence of Proposition~\ref{depth-sensitivity-result}.
\end{proof}

The next result was promised in (\ref{promise}).

\begin{corollary}\label{max-CM}If $k$ is a field, $\goth g\le \goth f$ are positive integers, $R$ is the polynomial ring$$R=k[\{x_{i,j}\mid 1\le j\le \goth g, 1\le i\le \goth f\}],$$ $\Phi:R^{\goth g}\to R^{\goth f}$ is the generic map given by 
$$\bmatrix x_{1,1}&\dots&x_{1,\goth g}\\\vdots&&\vdots\\
 x_{\goth f,1}&\dots&x_{\goth f,\goth g}\endbmatrix,$$
and $M=\HH_0(\mathcal C^{i,a}_\Phi)$ for some $i$ and $a$ with $-1\le i\le \goth f-\goth g$ and $1\le a\le \goth g$, then $\ann_R(M)=I_\goth g(\Phi)$ and $M$ is a maximal Cohen-Macaulay $(R/I_\goth g(\Phi))$-module of rank $\binom{\goth g-1}{a-1}$.
\end{corollary}
\begin{proof}The ring $R$ is Cohen-Macaulay and the $R$-module $M$ is perfect (by Theorem~\ref{main}.\ref{2.a}); hence $M$ is a Cohen-Macaulay $R$-module.
If $\mathfrak p$ is in  $\Ass_RM$, then $I_g(\Phi)\subseteq \ann(M)\subseteq \mathfrak p$
(see Observation~\ref{ann(H0)})  and $\goth f-\goth g+1= \grade \mathfrak p$ (by \cite[1.4.15]{BH}). On the other hand, $I_\goth g(\Phi)$ is already a prime ideal of $R$ of grade $\goth f-\goth g+1$. It follows that
$I_\goth g(\Phi)=\mathfrak p$,
  $I_\goth g(\Phi)=\ann_R M$, $\Supp_R(M)=\Supp_R(R/I_g(\Phi))$, and $\Ass_RM=\{I_g(\Phi)\}$. \end{proof}

\begin{remark}A module over a local Artinian ring has rank only if it is free.  Example~\ref{5.6} shows that the hypothesis $\goth f-\goth g+2\le \grade I_{\goth g-1}(\Phi)$ is needed in Theorem~\ref{main}.\ref{emain.d'}. Indeed, if $R=k[x,y]$, for some field $k$, and $$\Phi=\bmatrix x&0&0\\y&x&0\\0&y&x\\0&0&y\endbmatrix,$$ then $\HH_0(\mathcal C^{0,2}_\Phi)$ is 
an $R/I_3(\Phi)$-module which does not have any rank. 
It is easy to see that the length of $\HH_0(\mathcal C^{0,2}_\Phi)$ is twice the length of 
  $R/I_3(\Phi)=k[x,y]/(x,y)^3$; however, $\HH_0(\mathcal C^{0,2}_\Phi)$ has the wrong Betti numbers, as a module over $R$, to be a free  $R/I_3(\Phi)$-module. 
\end{remark}

\end{document}